\documentclass[11pt]{siamltex}
\usepackage{amsmath}
\usepackage{amsfonts}
\usepackage{amssymb}
\usepackage{graphicx}
\usepackage{fullpage}
\providecommand{\U}[1]{\protect\rule{.1in}{.1in}}

\newtheorem{algorithm}[theorem]{Algorithm}

\newtheorem{assumption}[theorem]{Assumption}
\begin{document}

\title{On the Convergence of a Non-linear Ensemble Kalman Smoother}
\author{Elhoucine Bergou\footnotemark[1]
\and Serge Gratton\footnotemark[1]
\and Jan Mandel\footnotemark[2]
}
\maketitle

\begin{abstract}
Ensemble methods, such as the ensemble Kalman filter (EnKF), the local
ensemble transform Kalman filter (LETKF), and the ensemble Kalman smoother
(EnKS) are widely used in sequential data assimilation, where state vectors
are of huge dimension. Little is known, however, about the asymptotic behavior of
ensemble methods. In this paper, 
we prove convergence in $L^{p}$ of ensemble Kalman smoother  to the Kalman smoother in the large-ensemble limit, as well as the convergence
of EnKS-4DVAR, which is a Levenberg-Marquardt-like algorithm with EnKS as the
linear solver, to the classical Levenberg-Marquardt algorithm in which the  
linearized problem is solved exactly.

\end{abstract}

%

\newcommand{\slugmaster}{\slugger{juq}{xxxx}{xx}{x--x}}%


\renewcommand{\thefootnote}{\fnsymbol{footnote}} \footnotetext[1]{INP-ENSEEIHT
and CERFACS, Toulouse, France.} \footnotetext[2]{University of Colorado
Denver, Denver, CO, USA, and Institute of Computer Science, Academy of
Sciences of the Czech Republic, Prague, Czech Republic. Partially supported by
the U.S. National Science Foundation under the grant DMS-1216481, the Czech
Science Foundation under the grant 13-34856S and the Fondation STAE project ADTAO.}


\section{Introduction}

Data assimilation is the process of blending  estimates of a given
system state, in the form of observational information and a prior knowledge
\cite{LeDimet-1986-VAA}. The Kalman filter/smoother (KF/KS)
\cite{Burgers-1998-ASE, Evensen-2009-DAE,Kalnay-2010-EKF} and the three and
four-dimensional variational assimilation system (3DVAR/4DVAR)
\cite{Courtier-1994-SOI, Tremolet-2007-ME4} are among well-known algorithms
used in data assimilation. Kalman filters estimate the state sequentially by
seeking an analysis that minimizes the posterior variance, while the 3DVAR and
4DVAR methods produce posterior maximum likelihood solutions through
minimization of an objective function. For high-dimensional problems, the
ensemble Kalman filter/smoother (EnKF/EnKS)
\cite{Khare-2008-IAE,Evensen-2009-DAE} and their variants have been proposed
as Monte Carlo derivative-free alternatives to the KF and KS, with the
intractable state covariance in the KF or in the KS replaced by the sample
covariance computed from an ensemble of realizations.

The purpose of this paper is to provide theoretical results for the method
originally proposed in \cite{Mandel-2013-4EK}, called EnKS-4DVAR. The
EnKS-4DVAR method uses an ensemble Kalman smoother as a linear solver in the
Gauss-Newton or Levenberg-Marquardt method to minimize the weak-constraint
4DVAR objective function. Further details on implementation and computational
results can be found in \cite{Mandel-2013-4EK}.

The equivalence of the Kalman smoother and incremental variational data
assimilation has been known for a long time; see, e.g.,
\cite{Bell-1993-IKF,Li-2001-OVD}. Hybridization of variational and
ensemble-based methods has been a topic of interest among researchers in
recent years \cite{Hamill-2000-HEK-x, Zupanski-2005-MLE, Wang-2010-IEC,
Sakov-2012-IES, Bocquet-2012-CII, Bocquet-2014-IEK}. The maximum ensemble
likelihood filter (MELF) \cite{Zupanski-2005-MLE} uses repeated EnKF on the
tangent problem to minimize the objective function over the span of the
ensemble. The iterated ensemble Kalman filer (IEnKF) \cite{Sakov-2012-IES}
solves the Euler equations for the minimum by Newton's method, preconditioned
by a square root ensemble Kalman filter, while \cite{Bocquet-2012-CII} adds a
regularization term, similar to the Levenberg-Marquardt method, and
\cite{Bocquet-2014-IEK} extends the IEnK method to strong-constraint 4DVAR.
The IEnKF uses a scaling of the ensemble, called the \textquotedblleft bundle
variant" to approximate the derivatives (tangent operators), achieving a
similar effect as the use of finite differences here. The four-dimensional 
ensemble-based variational data assimilation (4DEnVar) of \cite{Liu-2008-EFV,
Liu-2009-EFV, Liu-2013-EFV} minimizes the 4DVAR objective function over the
span of the ensemble.


Usually, in the formulation of the ensemble based methods (EnKF/EnKS and their
variants), each ensemble member is considered as a vector in $\mathbf{R}^{n}$,
that is, each vector is regarded as a sample point of a random vector. In this
paper,
we investigate a different way to interpret such algorithms,
similarly as in \cite{Mandel-2011-CEK, LeGland-2011-LSA}, namely, each
ensemble member is considered as a random vector and not merely as vector of
$\mathbf{R}^{n}$.
In fact, the elements of the EnKF/EnKS can be seen as random vectors instead of their 
realizations. Surprisingly, in this case little is known about the asymptotic
behavior of the EnKF/EnKS and other related ensemble methods.
 This is in contrast to particle filters, for which the asymptotic behavior as the number of particles increases to infinity is well studied. 
An important question related to EnKF/EnKS and related ensemble methods is a
law of large numbers-type theorem as the size of the ensemble grows to
infinity. In \cite{Mandel-2011-CEK, LeGland-2011-LSA}, it was proved that the
ensemble mean and covariance of EnKF converge to those of the KF, as the
number of ensemble members grows to infinity, but the convergence results are
not dimension independent. The analysis in \cite{Mandel-2011-CEK} relies on
the fact that ensemble members are exchangeable and uses the uniform
integrability theorem, which does not provide convergence rates; in
\cite{LeGland-2011-LSA}, stochastic inequalities for the random matrices and
vectors are used to obtain the classical rate $1/\sqrt{N}$, where $N$ is the
ensemble size, but it relies on entry-by-entry arguments. Convergence in
$L^{p}$ with the rate $1/\sqrt{N}$ independent of dimension
(including infinite) was obtained recently for the square root ensemble Kalman
filter \cite{Kwiatkowski-2014-CSR}. These analyses apply to each time step
separately rather than for the long-time behavior. The EnKF was proved to be
well-posed and to stay within a bounded distance from the truth, for a class
of dynamical systems, with the whole state observed, and when a sufficiently
large covariance inflation is used \cite{Kelly-2014-WAE}.

In this paper, we extend the convergence result of \cite{Mandel-2011-CEK} to
EnKS, and apply the extension to EnKS-4DVAR. 
 The randomness of the elements of
EnKS implies that, in contrast to the EnKS-4DVAR algorithm presented in
\cite{Mandel-2013-4EK}, the coefficients and the solution
of the linearized subproblem at each iteration are random. We investigate also the asymptotic
behavior of this algorithm.
We show the convergence of the EnKS to the KS in $L^{p}$ for all
$p\in[1,\infty)$ in the large ensemble limit, in the sense that the ensemble
mean and covariance constructed by EnKS method converge to the mean and
covariance of the KS respectively in $L^{p}$.
Finally, we show the convergence of the EnKS-4DVAR iterates to their
corresponding iterates in the classical Levenberg-Marquardt algorithm. Since the 
EnKS-4DVAR algorithm uses finite differences for approximating derivatives,
(i) we start by showing the convergence in probability of its iterates to
the iterates generated by the algorithm with exact derivatives as the finite
differences parameter goes to zero, (ii) then we prove the convergence in
$L^{p}$ of its iterates as the size of the ensemble grows to infinity.

The paper is organized as follows: in Section \ref{sec:preliminaries} we
recall some definitions and preliminary results that will be useful throughout the paper. Section \ref{sec:problem} introduces nonlinear data
assimilation. Section \ref{sec:kalman-fltering} contains the statements of the
KF and the EnKF, and recalls the convergence properties of the EnKF as the
ensemble size increases to infinity. Section \ref{sec:KS_EnKS} gives
statements of the KS and the EnKS, and extends the convergence properties of
the EnKF as the ensemble size goes to infinity, to the EnKS. Finally, Section
\ref{sec:vda-4DVAR} recalls the EnKS-4DVAR algorithm and presents the
convergence properties.

\section{Preliminaries}

\label{sec:preliminaries}

We recall definition of sequence of random vectors exchangeability, the notion
of convergence in probability and in $L^{p}$ of random elements. Then we
present several lemmas, which will be useful for the following of the paper.

\begin{definition}
[Exchangeability of random vectors] A set of $N$ random vectors $[X^{1}%
,\ldots,X^{N}]$ is exchangeable if their joint distribution is invariant to a
permutation of the indices; that is, for any permutation $\pi$ of the numbers
$1,\ldots,N$ and any Borel set $B$,
\[
\mathbb{P}\left(  [X^{\pi(1)},\ldots,X^{\pi(N)}]\in B\right)  =\mathbb{P}%
\left(  [X^{1},\ldots,X^{N}]\in B\right)  .
\]

\end{definition}

Clearly, an i.i.d sequence is exchangeable.

If $X$ is a random element (either vector or matrix), we use $|X|$ to denote
the usual Euclidean norm (for vectors) or spectral norm (for a matrix). For
$1\leq p<\infty$, denote
\[
\Vert X\Vert_{p}=E\left(  \left\vert X\right\vert ^{p}\right)  ^{1/p}.
\]
The space $L^{p}$ (of vectors or matrices) consists of all random elements $X$
(with values in the same space) such that the $E\left(  \left\vert
X\right\vert ^{p}\right)  <\infty$. Identifying random elements equal a.s., we
have that $\Vert.\Vert_{p}$ is a norm on the space $L^{p}$. Convergence in
$L^{p}$ is defined as the convergence in this norm. Note that if the element
$X$ is deterministic,%
\[
\Vert X\Vert_{p}=E\left(  \left\vert X\right\vert ^{p}\right)  ^{1/p}=\left(
\left\vert X\right\vert ^{p}\right)  ^{1/p}=\left\vert X\right\vert .
\]

\begin{definition}
[Convergence in probability] A sequence $(X^{k})$ of random vectors converges
in probability towards the random vector $X$ if for all $\epsilon>0$,
\[
\lim_{k\rightarrow\infty}\mathbb{P}\left(  \left\vert X^{k}-X\right\vert
\geq\epsilon\right)  =0,
\]
i.e.,%
\[
\forall\epsilon>0\ \forall\tilde{\epsilon}>0\ \exists k_{0}\ \forall k\geq
k_{0}:\mathbb{P}\left[  \left\vert X^{k}-X\right\vert \leq\epsilon\right]
\geq1-\tilde{\epsilon}.
\]
Convergence in probability will be denoted by%
\[
X^{k}%
\xrightarrow{\mathrm{P}}%
X\text{ as }k\rightarrow\infty.
\]
The concept of convergence in probability and the notation are extended in an
obvious manner to the case when the random vectors are indexed by $\tau>0$.
Then%
\[
X^{\tau}%
\xrightarrow{\mathrm{P}}%
X\text{ as }\tau\rightarrow0
\]
means%
\[
\forall\varepsilon>0\ \forall\tilde{\varepsilon}>0\ \exists\tau_{0}%
>0\ \forall0<\tau<\tau_{0}:\mathbb{P}\left[  \left\vert X^{\tau}-X\right\vert
\leq\varepsilon\right]  \geq1-\tilde{\varepsilon}.
\]

\end{definition}

We state the following lemmas, which will be used in this paper.

\begin{lemma}
\label{lem:exchangeability_sum} If random elements $Y^{1},\ldots,Y^{N}$ are
exchangeable, and $Z^{1},\ldots,Z^{N}$ are also exchangeable, and independent
from $Y^{1},\ldots,Y^{N}$, then $Y^{1}+Z^{1},\ldots,Y^{N}+Z^{N}$ are exchangeable.
\end{lemma}

\begin{lemma}
\label{lem:exchangeability_F} If random elements $Y^{1},\ldots,Y^{N}$ are
exchangeable, and
\[
Z^{k}=F\left(  Y^{1},\ldots,Y^{N},Y^{k}\right)  ,
\]
where $F$ is measurable and permutation invariant in the first $N$ arguments,
then $Z^{1},\ldots,Z^{N}$ are also exchangeable.
\end{lemma}

For the proof of the previous two lemmas, we refer to \cite{Mandel-2011-CEK}.

\begin{lemma}
[Uniform integrability] If $(X^{k})$ is a bounded sequence in $L^{p}$ and
$X^{k}%
\xrightarrow{\mathrm{P}}%
X$, then $\Vert X^{k}-X\Vert_{q}\rightarrow0$ for all $1\leq q<p$.
\end{lemma}

\begin{proof}
The proof is an exercise on uniform integrability \cite[page 338]
{Billingsley-1995-PM}: Let $1\leq q<p$. The sequence $\left(  E\left(
\left\vert X^{k}-X\right\vert ^{q\left(  p/q\right)  }\right)  \right)  $ is
bounded and $p/q>1$, thus the sequence $\left(  \left\vert X^{k}-X\right\vert
^{q}\right)  $ is uniformly integrable. Since $\left\vert X^{k}-X\right\vert
\xrightarrow{\mathrm{P}}%
0$, and thus $\left\vert X^{k}-X\right\vert ^{q}%
\xrightarrow{\mathrm{P}}%
0$, it follows that $E\left(  \left\vert X^{k}-X\right\vert ^{q}\right)
\rightarrow0$.
\end{proof}

\begin{lemma}
[Continuous mapping theorem] Let $X^{k}$ be a sequence of random elements with
values on a metric space $\mathcal{A}$, such that $X^{k}%
\xrightarrow{\mathrm{P}}%
X$. Let $f$ be a continuous function from $\mathcal{A}$ to another metric
space $\mathcal{B}$. Then $f(X^{k})%
\xrightarrow{\mathrm{P}}%
f(X)$.
\end{lemma}

We refer to \cite[Theorem 2.3]{vanderVaart-2000-AS} for a proof.

\section{The nonlinear data assimilation problem}


\label{sec:problem}%

\begin{table}[tb] \centering
\begin{tabular}
[c]{llll}%
Symbol & Random & Meaning \hspace{6.65cm} First used in & \hspace{-0.48cm}
Sec.\\\hline
$X_{i}$ & yes & the state at time $i$ & \ref{sec:problem}\\
$X_{i|\ell}$ & no & the mean of $X_{i}$ given data $y_{1:\ell}$ &
\ref{sec:filter}\\
$P_{i|\ell}$ & no & the covariance of $X_{i}$ given data $y_{1:\ell}$ &
\ref{sec:filter}\\
$X_{i|\ell}^{n}$ & yes & member $n$ of an ensemble approximating $X_{i}$ given
$y_{1:\ell}$ & \ref{sec:EnKF}\\
$\bar{X}_{i|\ell}^{N}$ & yes & the sample mean of the ensemble $X_{i|\ell}%
^{1},\ldots,X_{i|\ell}^{N}$ & \ref{sec:EnKF}\\
${P}_{i|\ell}^{N}$ & yes & the sample covariance of the ensemble $X_{i|\ell
}^{1},\ldots,X_{i|\ell}^{N}$ & \ref{sec:EnKF}\\
$U_{i|\ell}^{n}$ & yes & member $n$ of a reference ensemble approximating
$X_{i}$ given $y_{1:\ell}$ & \ref{sec:EnKF-convergence}\\
$X_{0:i}$ & yes & composite state $\left[  X_{0},\ldots,X_{i}\right]  $ at
times $0,\ldots,i$ & \ref{sec:smoother}\\
$X_{0:i|\ell}$ & no & the mean of $X_{0:i}$ & \ref{sec:smoother}\\
${P}_{0:i|\ell}$ & no & the covariance of $X_{0:i}$ & \ref{sec:smoother}\\
$X_{0:i|\ell}^{n}$ & yes & member $n$ of ensemble approximating $X_{i}$ given
$y_{1:\ell}$ & \ref{sec:EnKS}\\
$\bar{X}_{0:i|\ell}^{N}$ & yes & the sample mean of the ensemble $X_{0:i|\ell
}^{1},\ldots,X_{0:i|\ell}^{N}$ & \ref{sec:EnKS}\\
${P}_{0:i,0:i|\ell}^{N}$ & yes & the sample covariance of the ensemble
$X_{0:i|\ell}^{1},\ldots,X_{0:i|\ell}^{N}$ & \ref{sec:EnKS}\\
$U_{0:i|\ell}^{n}$ & yes & member $n$ of reference ensemble approximating
$X_{i}$ given $y_{1:\ell}$ & \ref{sec:EnKS_convergence}\\
$x_{\mathrm{b}}$ & no & the background state & \ref{sec:4DVAR}\\
$x_{i}$ & no & the unknown state in 4DVAR minimization & \ref{sec:4DVAR}\\
$x_{0:k}$ & no & the unknown composite state in the 4DVAR minimization &
\ref{sec:4DVAR}\\
$x_{i}^{j},x_{0:k}^{j}$ & no & the iterate $j$ in the 4DVAR minimization &
\ref{sec:incremental-4DVAR}\\
$X_{0:i|\ell}^{j,n}$ & yes & member $n$ of ensemble approximating $x_{0:i}%
^{j}$ & \ref{sec:EnKS-4DVAR-WD}\\
$\bar{X}_{0:i|\ell}^{j,N}$ & yes & the sample mean of the ensemble
$X_{0:i|\ell}^{j,1},\ldots,X_{0:i|\ell}^{j,N}$ using & \ref{sec:EnKS-4DVAR-WD}%
\\
$X_{0:k|k}^{j,n,N_{j}}$ & yes & member $n$ from ensemble of size $N_{j}$ &
\ref{sec:EnKS-4DVAR-WD}\\
$X_{0:i|\ell}^{j,n,\tau}$ & yes & member $n$ of the ensemble approximating
$x_{0:i}^{j}$ with step $\tau$ & \ref{sec:EnKS-4DVAR-fd}\\
$\bar{X}_{0:i|\ell}^{j,n,\tau}$ & yes & the sample mean of the ensemble
$X_{0:i|\ell}^{j,1,\tau},\ldots,X_{0:i|\ell}^{j,N,\tau}$ &
\ref{sec:EnKS-4DVAR-fd}\\
$\Delta_{i|\ell}^{j,n}$ & yes & 4DVAR increment ensemble members $X_{i|\ell
}^{j,n}-x_{i}^{j-1,n}$ & \ref{sec:EnKS-4DVAR-fd}%
\end{tabular}
\caption{Notation for state vectors.}\label{tab:notation}%
\end{table}%

Consider the following classical system of stochastic equations with additive
Gaussian noise, which appears in different fields, such as weather forecasting
and hydrology,%
\begin{align}
X_{0}  &  \sim N(x_{\mathrm{b}},B)\label{eq:filter_background_eq}\\
X_{i}  &  =\mathcal{M}_{i}(X_{i-1})+\mu_{i}+V_{i},\quad V_{i}\sim
N(0,Q_{i}),\quad i=1,\ldots k\label{eq:filter_model_eq}\\
y_{i}  &  =\mathcal{H}_{i}(X_{i})+W_{i},\quad W_{i}\sim N(0,R_{i}),\quad
i=1,\ldots k, \label{eq:filter_observation_eq}%
\end{align}
with independent perturbations $V_{i}$ and $W_{i}$. The operators
$\mathcal{M}_{i}$ and $\mathcal{H}_{i}$ are the model operators and the
observation operators, respectively, and they are assumed to be continuously
differentiable. When they are linear, we denote them by $M_{i}$ and $H_{i}$,
respectively. The index $i$ denotes the time index and $k$ denotes the number
of time steps. While the outputs $y_{i}$ are observed, the state $X_{i}$ and
the noise variables $V_{i}$ and $W_{i}$ are hidden.
 The quantities $B$, $Q_i$ and $R_i$ are the covariance matrices of $X_0$, $V_i$ and $W_i$ respectively. The quantity $\mu_{i}$
 is a deterministic vector. 
 The objective is to
estimate the hidden states $X_{1},\ldots,X_{k}$.

\begin{definition}
The distribution of $X_{k}$ from (\ref{eq:filter_background_eq}%
)--(\ref{eq:filter_observation_eq}) conditioned on $y_{1},\ldots,y_{k-1}$ is
called prior distribution. The filtering, or posterior, distribution is the
distribution of $X_{k}$, conditioned on the observations of the data
$y_{1},\ldots y_{k}$. The smoothing distribution is the joint distribution of
$X_{0},\ldots,X_{k}$, conditioned on the observations of data $y_{1},\ldots
y_{k}$.
\end{definition}

In geosciences, the prior is usually called forecast and the posterior is
called analysis. In Table \ref{tab:notation}, we collect the notation for
state vectors and their ensembles for reference.


\section{Kalman filtering}

\label{sec:kalman-fltering} \nopagebreak

\subsection{Kalman filter}

\label{sec:filter}

The Kalman filter \cite{Kalman-1960-NAL} provides an efficient computational
recursive means to estimate the state of the process $X_{k}$ in the linear
case, i.e., when $\mathcal{M}_{i}$ and $\mathcal{H}_{i}$, $i=1,\ldots,k,$ are
linear. Denote the mean and the covariance of $X_{i}$ given the data
$y_{1},\ldots,y_{\ell}$, by
\[
X_{i|\ell}=E(X_{i}|y_{1},\ldots,y_{\ell}),\quad P_{i|\ell}=P(X_{i}%
|y_{1},\ldots,y_{\ell}),
\]
respectively. In the linear case, the probability distribution of the process
$X_{k}$ given the data up to the time $k$ is Gaussian, therefore it is
characterized by its mean and covariance matrix, which can be computed as follows.

\begin{algorithm}
[Kalman filter]\label{alg:KF} For $i=0$, set $X_{0|0}=x_{\mathrm{b}}$ and
$P_{0|0}=B$. For $i=1,\ldots,k,$%
\begin{align}
X_{i|i-1}  &  =M_{i}X_{i-1|i-1}+\mu_{i},\text{ (advance the mean in
time)}\label{eq:Kalman-model}\\
P_{i|i-1}  &  =M_{i}P_{i-1|i-1}M_{i}^{\mathrm{T}}+Q_{i},\text{ (advance the
covariance in time)}\nonumber\\
K_{i}  &  =P_{i|i-1}H_{i}^{\mathrm{T}}(H_{i}P_{i|i-1}H_{i}^{\mathrm{T}}%
+R_{i})^{-1}\text{ (the Kalman gain)}\nonumber\\
X_{i|i}  &  =X_{i|i-1}+K_{i}(y_{i}-H_{i}X_{i|i-1}),\text{ (update the mean
from the observation }i\text{)}\label{eq:Kalman-update-mean}\\
P_{i|i}  &  =(I-K_{i}H_{i})P_{i|i-1}\text{ (update the covariance from the
observation }i\text{)} \label{eq:Kalman-update-covariance}%
\end{align}

\end{algorithm}

In atmospheric sciences, the update (\ref{eq:Kalman-update-mean}%
)--(\ref{eq:Kalman-update-covariance}) is referred to as the analysis step.

\begin{lemma}
\label{lem:KF}The distribution $N\left(  X_{k|k},P_{k|k}\right)  $ from the
Kalman filter is the filtering distribution.

\end{lemma}

See, e.g., \cite{Anderson-1979-OF,Simon-2006-OSE} for the proof.

If the dimension of the hidden state $X_{k}$ is large, the covariance matrices
$P_{k|k-1}$ and $P_{k|k}$ are large dense matrices, hence storing such
matrices in memory with the current hardware is almost impossible, and the
matrix products in the computation of $P_{k|k-1}$ are also problematic. To
solve these problems, the idea is to use ensemble methods.

\subsection{Ensemble Kalman filter (EnKF)}

\label{sec:EnKF}The idea behind the ensemble Kalman filter is to use Monte
Carlo samples and the corresponding empirical covariance matrix instead of the
forecast covariance matrix $P_{k|k-1}$ \cite{Evensen-2009-DAE}. Denote by $n$
the ensemble member index, $n=1,\ldots,N$.

\begin{algorithm}
[EnKF] For $i=0$, $X_{0|0}^{n}\sim N\left(  x_{\mathrm{b}},B\right)  $. For
$i=1,\ldots,k$, given an analysis ensemble $X_{i-1|i-1}^{1},\ldots
,X_{i-1|i-1}^{N}$ at time $i-1$, the ensemble at time $i$ is built as
\begin{align}
X_{i|i-1}^{n}  &  ={M}_{i}X_{i-1|i-1}^{n}+\mu_{i}+V_{i}^{n},\quad V_{i}%
^{n}\sim N\left(  0,{Q}_{i}\right)  ,\label{eq:advance_model}\\
X_{i|i}^{n}  &  =X_{i|i-1}^{n}+{P}_{i|i-1}^{N}{H}_{i}^{\mathrm{T}}\left(
{H}_{i}{P}_{i|i-1}^{N}{H}_{i}^{\mathrm{T}}+{R}_{i}\right)  ^{-1}(y_{i}%
-W_{i}^{n}-H_{i}X_{i|i-1}^{n})\quad W_{i}^{n}\sim N\left(  0,{R}_{i}\right)  ,
\label{eq:ens_analysis}%
\end{align}
where ${P}_{i|i-1}^{N}$ is the covariance estimate from the ensemble $\left[
X_{i|i-1}^{n}\right]  _{n=1}^{N}$,
\[
{P}_{i|i-1}^{N}=\frac{1}{N-1}\sum_{n=1}^{N}\left(  X_{i|i-1}^{n}-\bar
{X}_{i|i-1}^{N}\right)  \left(  X_{i|i-1}^{n}-\bar{X}_{i|i-1}^{N}\right)
^{\mathrm{T}},\text{ where }\bar{X}_{i|i-1}^{N}=\frac{1}{N}\sum_{n=1}%
^{N}X_{i|i-1}^{n}.
\]
The empirical covariance matrix ${P}_{i|i-1}^{N}$ is never computed or stored,
indeed to compute the matrix products ${P}_{i|i-1}^{N}H_{i}^{\mathrm{T}}$ and
$H_{i}{P}_{i|i-1}^{N}H_{i}^{\mathrm{T}}$ only matrix-vector products are
needed:
\begin{align}
{P}_{i|i-1}^{N}H_{i}^{\mathrm{T}}  &  =\frac{1}{N-1}\sum_{n=1}^{N}\left(
X_{i|i-1}^{n}-\bar{X}_{i|i-1}^{N}\right)  \left(  X_{i|i-1}^{n}-\bar
{X}_{i|i-1}^{N}\right)  ^{\mathrm{T}}H_{i}^{\mathrm{T}}\label{eq:PHt}\\
&  =\frac{1}{N-1}\sum_{n=1}^{N}\left(  X_{i|i-1}^{n}-\bar{X}_{i|i-1}%
^{N}\right)  h_{n}^{\mathrm{T}},\nonumber\\
H_{i}{P}_{i|i-1}^{N}H_{i}^{\mathrm{T}}  &  =H_{i}\frac{1}{N-1}\sum_{n=1}%
^{N}\left(  X_{i|i-1}^{n}-\bar{X}_{i|i-1}^{N}\right)  \left(  X_{i|i-1}%
^{n}-\bar{X}_{i|i-1}^{N}\right)  ^{\mathrm{T}}H_{i}^{\mathrm{T}}=\frac{1}%
{N-1}\sum_{n=1}^{N}h_{n}h_{n}^{\mathrm{T}}, \label{eq:HPHt}%
\end{align}
where
\begin{equation}
h_{n}=H_{i}\left(  X_{i|i-1}^{n}-\bar{X}_{i|i-1}^{N}\right)  . \label{eq:h_n}%
\end{equation}

\end{algorithm}

Note that the i.i.d. random vectors $(V_{i}^{1},\ldots,V_{i}^{N})$ are
simulated here with the same statistics as the additive Gaussian noise $V_{i}$
in the original state in eq.~(\ref{eq:filter_model_eq}). The i.i.d. random
vectors $(W_{i}^{1},\ldots,W_{i}^{N})$ are simulated here with the same
statistics as the additive Gaussian noise $W_{i}$ in the original state in
eq.~(\ref{eq:filter_observation_eq}). The initial ensemble $\left[
X_{0|0}\right]_{n=1}^{N}$ is simulated as i.i.d. Gaussian random vectors
with mean $x_{\mathrm{b}}$ and covariance $B$, i.e. with the same statistics
as the initial state $X_{0}$.

\subsection{Convergence of the EnKF}

\label{sec:EnKF-convergence}For theoretical purposes, we define an auxiliary
ensemble $U_{i|i}=[U_{i|i}^{n}]_{n=1}^{N}$, $i=0,\ldots,k,$ called the
\emph{reference ensemble}, in the same way as the ensemble $X_{i|i}%
=[X_{i|i}^{n}]_{n=1}^{N}$, but this time for the updates of the ensemble
$U_{i|i}$ we use the exact covariances instead of their empirical estimates.
The realizations of the random perturbations $V_{i}^{n}$ and $W_{i}^{n}$ in
both ensembles are the same. Thus, for $i=0$, $U_{0|0}^{n}=X_{0|0}^{n}$ and
for $i=1,\ldots,k$, we build ${U}_{i|i}$ up to time $i$ conditioned on
observations up to time $i$,%
\begin{align}
U_{i|i-1}^{n}  &  ={M}_{i}U_{i-1|i-1}^{n}+\mu_{i}+V_{i}^{n},\quad V_{i}%
^{n}\sim N\left(  0,{Q}_{i}\right)  ,\quad n=1,\ldots,N,\label{eq:U_forcast}\\
U_{i|i}^{n}  &  =U_{i|i-1}^{n}+{P}_{i|i-1}{H}_{i}^{\mathrm{T}}\left(  {H}%
_{i}{P}_{i|i-1}{H}_{i}^{\mathrm{T}}+{R}_{i}\right)  ^{-1}\left(  y_{i}%
-W_{i}^{n}-H_{i}U_{i|i-1}^{n}\right)  , \label{eq:U_analysis}%
\end{align}
where $W_{i}^{n}\sim N\left(  0,{R}_{i}\right)  $ is a random perturbation,
and $P_{i|i-1}$ is the covariance of $U_{i|i-1}^{1}$
\[
{P}_{i|i-1}=E\left[  \left(  U_{i|i-1}^{n}-E(U_{i|i-1}^{n})\right)  \left(
U_{i|i-1}^{n}-E(U_{i|i-1}^{n})\right)  ^{\mathrm{T}}\right]  .
\]
Note that the only difference between the two ensembles $X_{k|k}$ and
$U_{k|k}$ is that for the construction of $X_{k|k}$, we use the empirical
prediction covariance ${P}_{k|k-1}^{N}$ of the ensemble, which depends on all
ensemble members, instead of the exact covariance. Therefore, $X_{k|k}^{n}$,
$n=1,\ldots,N$, are in general dependent. On the other hand:

\begin{lemma}
\label{lem:EnKF-Uk-exact} The members of the ensemble $[U_{k|k}^{n}]_{n=1}^{N}$
are i.i.d and the distribution of each $U_{k|k}^{n}$ is the same as the
filtering distribution.
\end{lemma}

\begin{proof}
The proof is by induction and the same as in \cite[Lemma 4]{Mandel-2011-CEK},
except we take the additional perturbation $V_{k}^{n}$ into account. Since
$\left[  V_{k}^{n}\right]  _{n=1}^{N}$ are Gaussian and independent of
everything else by assumption, $[U_{k|k}^{n}]_{n=1}^{N}$ are independent and
Gaussian. The forecast covariance ${P}_{k|k-1}$ is constant (non-random), and,
consequently, the analysis step (\ref{eq:U_analysis}) is a linear
transformation, which preserves the independence of the ensemble members and
the Gaussianity of the distribution. It is known that the members of the
reference ensemble have the same mean and covariance as given by the Kalman
filter \cite[eq.~(15)\ and (16)]{Burgers-1998-ASE}. The proof is completed by
noting that a Gaussian distribution is determined by its mean and covariance.
\end{proof}

\begin{theorem}
\label{thm:EnKF_convergence} For any $i=0,\ldots,k$, the random matrix
\begin{equation}
\left[
\begin{array}
[c]{c}%
X_{i|i}^{1},\ldots,X_{i|i}^{N}\\
U_{i|i}^{1},\ldots,U_{i|i}^{N}%
\end{array}
\right]  \label{eq:exch}%
\end{equation}
has exchangeable columns, and%
\[
X_{i|i}^{1}\rightarrow U_{i|i}^{1},
\]
in all $L^{p}$, $1\leq p<\infty$, as $N\rightarrow\infty$. Also,
\begin{align*}
\bar{X}_{i|i-1}^{N}  &  =\frac{1}{N}\sum_{n=1}^{N}X_{i|i}^{n}\rightarrow
E\left(  U_{i|i}^{1}\right)  ,\\
{P}_{i|i-1}^{N}  &  =\frac{1}{N-1}\sum_{n=1}^{N}\left(  X_{i|i-1}^{n}-\bar
{X}_{i|i-1}^{N}\right)  \left(  X_{i|i-1}^{n}-\bar{X}_{i|i-1}^{N}\right)
^{\mathrm{T}}\\
&  \rightarrow{P}_{i|i-1}=E\left[  \left(  U_{i|i-1}^{1}-E\left(
U_{i|i-1}^{1}\right)  \right)  \left(  U_{i|i-1}^{1}-E\left(  U_{i|i-1}%
^{1}\right)  \right)  ^{\mathrm{T}}\right]  ,
\end{align*}
in all $L^{p}$, $1\leq p<\infty$, as $N\rightarrow\infty$.
\end{theorem}

\begin{proof}
The theorem is again a simple extension of that of \cite[Theorem
1]{Mandel-2011-CEK}, by adding the model error $V_{i}^{n}$ in each step of the
induction over $i$.
\end{proof}

Note that since (\ref{eq:exch}) has exchangeable columns and $X_{i|i}%
^{1}\rightarrow U_{i|i}^{1}$ in $L^{p}$, we have the same convergence result
for every fixed $n$, $X_{i|i}^{n}\rightarrow U_{i|i}^{n}$ in all $L^{p}$, as
$N\rightarrow\infty$.

\section{Kalman smoothing}

\label{sec:KS_EnKS}

\subsection{Kalman smoother (KS)}

\label{sec:smoother}
A smoother estimates the composite hidden state
\[
X_{0:i}=\left[
\begin{array}
[c]{c}%
X_{0}\\
\vdots\\
X_{i}%
\end{array}
\right]
\]
given all observations $y_{1},\ldots,y_{i}$. Again, the Kalman smoother
provides the exact result in the linear Gaussian case. Denote by $X_{0:i|\ell}$
the expectation of the composite state $X_{0:i}$ given the observations
$y_{1},\ldots,y_{\ell}$, and by $P_{0:i|\ell}$ the corresponding covariance.
In the linear case, we write the stochastic system (\ref{eq:filter_model_eq}%
)--(\ref{eq:filter_observation_eq}) in terms of the composite state $X_{0:i}$
as
\begin{align}
X_{0:i}  &  =\left[
\begin{array}
[c]{cccc}%
I_{m} & 0 & \ldots & 0\\
0 & I_{m} & \vdots & \vdots\\
\vdots & \ddots & \ddots & 0\\
0 & \ldots & \ddots & I_{m}\\
0 & \ldots & 0 & M_{i}%
\end{array}
\right]  X_{0:i-1}+\left[
\begin{array}
[c]{c}%
0\\
\vdots\\
\mu_{i}%
\end{array}
\right]  +\left[
\begin{array}
[c]{c}%
0\\
\vdots\\
V_{i}%
\end{array}
\right] \label{eq:smoother_model}\\
&  =\left[
\begin{array}
[c]{c}%
I_{m(i-1)}\\
\tilde{M}_{i}%
\end{array}
\right]  X_{0:i-1}+\left[
\begin{array}
[c]{c}%
0\\
\vdots\\
\mu_{i}%
\end{array}
\right]  +\left[
\begin{array}
[c]{c}%
0\\
\vdots\\
V_{i}%
\end{array}
\right]  ,\quad V_{i}\sim N\left(  0,Q_{i}\right)  ,\nonumber\\
y_{i}  &  =\left[  0,\ldots,{H}_{i}\right]  X_{0:i}+W_{i}=\tilde{H}_{i}%
X_{0:i}+W_{i},\quad W_{i}\sim N\left(  0,R_{i}\right)  ,
\label{eq:smoother_obs}%
\end{align}
where $m$ is the dimension of the state $X_{i}$,
$I_{d}$ is the identity matrix in $\mathbf{R}^{d\times d}$, and
\begin{equation}
\tilde{H}_{i}=\left[  0,\ldots,{H}_{i}\right]  ,\text{\quad}\tilde{M}%
_{i}=\left[  0,\ldots,M_{i}\right]  . \label{eq:smoother_obs_op}%
\end{equation}
Applying the Kalman filter analysis step (\ref{eq:Kalman-model}%
)--(\ref{eq:Kalman-update-covariance}) to the observation
(\ref{eq:smoother_obs}) of the composite state $X_{0:i}$, we obtain the Kalman smoother:%

\begin{align*}
X_{0:i|i-1}  &  =\left[
\begin{array}
[c]{c}%
I_{m(i-1)}\\
\tilde{M}_{i}%
\end{array}
\right]  X_{0:i-1|i-1}+\left[
\begin{array}
[c]{c}%
0\\
\vdots\\
\mu_{i}%
\end{array}
\right]  =\left[
\begin{array}
[c]{c}%
X_{0:i-1|i-1}\\
M_{i}X_{i-1,i-1}+\mu_{i}%
\end{array}
\right]  ,\\
P_{0:i|i-1}  &  =\left[
\begin{array}
[c]{c}%
I_{m(i-1)}\\
\tilde{M}_{i}%
\end{array}
\right]  P_{0:i-1|i-1}\left[
\begin{array}
[c]{c}%
I_{m(i-1)}\\
\tilde{M}_{i}%
\end{array}
\right]  ^{\mathrm{T}}+\left[
\begin{array}
[c]{cc}%
0 & 0\\
0 & Q_{i}%
\end{array}
\right] \\
&  =\left[
\begin{array}
[c]{cc}%
P_{0:i-1|i-1} & P_{0:i-1|i-1}\tilde{M}_{i}^{\mathrm{T}}\\
\tilde{M}_{i}P_{0:i-1|i-1} & \tilde{M}_{i}P_{0:i-1|i-1}\tilde{M}%
_{i}^{\mathrm{T}}+Q_{i}%
\end{array}
\right]  ,\\
K_{i}  &  =P_{0:i|i-1}\tilde{H}_{i}^{\mathrm{T}}\left(  R_{i}+\tilde{H}%
_{i}P_{0:i|i-1}\tilde{H}_{i}^{\mathrm{T}}\right)  ^{-1}\\
&  =P_{0:i|i-1}\tilde{H}_{i}^{\mathrm{T}}\left(  R_{i}+{H}_{i}P_{i,i|i-1}%
{H}_{i}^{\mathrm{T}}\right)  ^{-1},\\
X_{0:i|i}  &  =X_{0:i|i-1}+K_{i}\left(  y_{i}-\tilde{H}_{i}X_{0:i|i-1}\right)
=X_{0:i|i-1}+K_{i}\left(  y_{i}-{H}_{i}X_{i|i-1}\right)  ,\\
P_{0:i|i}  &  =\left(  I_{mi}-K_{i}\tilde{H}_{i}\right)  P_{0:i|i-1}.
\end{align*}

\begin{lemma}
\label{lem:KS}The distribution $N\left(  X_{0:k|k},P_{0:k,0:k|k}\right)  $
from the Kalman smoother is the smoothing distribution, and its mean
$X_{0:k|k}$ is the solution of the least squares problem,%
\begin{equation}
\text{ }X_{0:k|k}=%
\mathop{\mathrm{argmin}}%
_{x_{0:k}}\biggl(\left\vert x_{0}-x_{\mathrm{b}}\right\vert _{B^{-1}}^{2}%
+\sum_{i=1}^{k}\left\vert x_{i}-{M}_{i}x_{i-1}-\mu_{i}\right\vert _{Q_{i}%
^{-1}}^{2}+\sum_{i=1}^{k}\left\vert y_{i}-{H}_{i}x_{i}\right\vert _{R_{i}%
^{-1}}^{2}\biggr). \label{eq:eq_ks_ls}%
\end{equation}

\end{lemma}

\begin{proof}
The mean $X_{0:k|k}$ maximizes the joint posterior probability density of the
composite state $X_{0:k}$ given $y_{1:k}$, which is proportional to
\[
e^{-\frac{1}{2}\left\vert x_{0}-x_{\mathrm{b}}\right\vert _{B^{-1}}^{2}%
}e^{-\frac{1}{2}\sum_{i=1}^{k}\left\vert x_{i}-{M}_{i}x_{i-1}-\mu
_{i}\right\vert _{Q_{i}^{-1}}^{2}}e^{-\frac{1}{2}\sum_{i=1}^{k}\left\vert
y_{i}-{H}_{i}x_{i}\right\vert _{R_{i}^{-1}}^{2}}%
\]
from the Bayes theorem.
\end{proof}

Again, when $m$ is large, the covariance matrices $P_{0:i|i-1}$ and
$P_{0:i|i}$ are very large and the matrix products in the computation of
$P_{0:i|i-1}$ is also problematic to implement, and we turn to ensemble methods.

\subsection{Ensemble Kalman smoother (EnKS)}

\label{sec:EnKS} In the ensemble Kalman smoother \cite{Evensen-2009-DAE}, the
covariances are replaced by approximations from the ensemble. Let
\[
\left[  \left[
\begin{array}
[c]{c}%
X_{0|j}^{1}\\
\vdots\\
X_{i|j}^{1}%
\end{array}
\right]  ,\ldots,\left[
\begin{array}
[c]{c}%
X_{0|j}^{N}\\
\vdots\\
X_{i|j}^{N}%
\end{array}
\right]  \right]  =\left[  X_{0:i|j}^{1},\ldots,X_{0:i|j}^{N}\right]  =\left[
X_{0:i|j}^{n}\right]  _{n=1}^{N}%
\]
denote an ensemble of $N$ model states over time up to $i$, conditioned on the
observations up to time $j$.

\begin{algorithm}
[EnKS]\label{alg:EnKS} For $i=0$, the ensemble $\left[  X_{0|0}^{n}\right]
_{n=1}^{N}$ consists of i.i.d. Gaussian random variables
\begin{equation}
X_{0|0}^{n}\sim N\left(  x_{\mathrm{b}},B\right)  . \label{eq:EnKS-background}%
\end{equation}
For $i=1,\ldots,k$, advance the model to time $i$ by%
\begin{equation}
X_{i|i-1}^{n}={M}_{i}X_{i-1|i-1}^{n}+\mu_{i}+V_{i}^{n},\quad V_{i}^{n}\sim
N\left(  0,{Q}_{i}\right)  ,\quad n=1,\ldots,N. \label{eq:EnKS-advance}%
\end{equation}
Incorporate the observation at time $i$,
\[
y_{i}=\tilde{H}_{i}X_{i}+W_{i},\quad W_{i}\sim N\left(  0,{R}_{i}\right)
\]
into the ensemble of composite states $\left[  X_{0:i|i-1}^{1},\ldots
,X_{0:i|i-1}^{N}\right]  $ in the same way as for the EnKF update,%
\begin{equation}
X_{0:i|i}^{n}=X_{0:i|i-1}^{n}+{P}_{0:i|i-1}^{N}\tilde{H}_{i}^{\mathrm{T}%
}\left(  \tilde{H}_{i}{P}_{0:i|i-1}^{N}\tilde{H}_{i}^{\mathrm{T}}+{R}%
_{i}\right)  ^{-1}\left(  y_{i}-W_{i}^{n}-H_{i}X_{i|i-1}^{n}\right)
\label{eq:EnKS-analysis}%
\end{equation}
where ${P}_{0:i|i-1}^{N}$ is a covariance estimate from the ensemble
$X_{0:i|i-1}$ and $W_{i}^{n}\sim N\left(  0,{R}_{i}\right)  $ are random
perturbations. Similarly as in (\ref{eq:PHt})--(\ref{eq:h_n}), only the
following matrix-vector products are needed:
\begin{align}
{P}_{0:i|i-1}^{N}\tilde{H}_{i}^{\mathrm{T}}  &  =\frac{1}{N-1}\sum_{n=1}%
^{N}\left(  X_{0:i|i-1}^{n}-\bar{X}_{0:i|i-1}^{N}\right)  \left(
X_{i|i-1}^{n}-\bar{X}_{i|i-1}^{N}\right)  ^{\mathrm{T}}H_{i}^{\mathrm{T}%
}\label{eq:PHt_EnKS}\\
&  =\frac{1}{N-1}\sum_{n=1}^{N}\left(  X_{0:i|i-1}^{n}-\bar{X}_{0:i|i-1}%
^{N}\right)  h_{n}^{\mathrm{T}},\nonumber\\
\tilde{H}_{i}^{\mathrm{T}}{P}_{0:i|i-1}\tilde{H}_{i}^{\mathrm{T}}  &
=H_{i}\frac{1}{N-1}\sum_{n=1}^{N}\left(  X_{i|i-1}^{n}-\bar{X}_{i|i-1}%
^{N}\right)  \left(  X_{i|i-1}^{n}-\bar{X}_{i|i-1}^{N}\right)  ^{\mathrm{T}%
}H_{i}^{\mathrm{T}}=\frac{1}{N-1}\sum_{n=1}^{N}h_{n}h_{n}^{\mathrm{T}},
\label{eq:HPHt_EnKS}%
\end{align}
where again
\begin{equation}
h_{n}=H_{i}\left(  X_{i|i-1}^{n}-\bar{X}_{i|i-1}^{N}\right)
\label{eq:h_n_EnKS}%
\end{equation}
and
\begin{equation}
\bar{X}_{\ell|i-1}^{N}=\frac{1}{N}\sum_{n=1}^{N}X_{\ell|i-1}^{n}.
\label{eq:sample_mean_EnKS}%
\end{equation}

\end{algorithm}

\subsection{Convergence of the EnKS}

\label{sec:EnKS_convergence} Just as for the EnKF, we construct an ensemble
$U_{0:k|k}=[U_{0:k|k}^{n}]_{n=1}^{N}$ in the same way as the ensemble
$\left[X_{0:k|k}^n \right]_{n=1}^N$, but for the updates of the ensemble $U_{0:k|k}$ we use the exact
covariances instead of their empirical estimates. So, for $i=0$, $U_{0|0}%
^{n}=X_{0|0}^{n}$, and for $i=1,\ldots,k$, $n=1,\ldots,N,$%
\begin{align*}
U_{i|i-1}^{n}  &  ={M}_{i}U_{i-1|i-1}^{n}+\mu_{i}+V_{i}^{n},\quad V_{i}%
^{n}\sim N\left(  0,{Q}_{i}\right)  ,\\
U_{0:i|i}^{n}  &  =U_{0:i|i-1}^{n}+{P}_{0:i|i-1}\tilde{H}_{i}^{\mathrm{T}%
}\left(  \tilde{H}_{i}{P}_{0:i|i-1}\tilde{H}_{i}^{\mathrm{T}}+{R}_{i}\right)
^{-1}\left(  y_{i}-W_{i}^{n}-H_{i}U_{i|i-1}^{n}\right)  ,\\
&  W_{i}^{n}\sim N\left(  0,{R}_{i}\right)  ,
\end{align*}
where ${P}_{0:i|i-1}$ is the covariance of $U_{0:i|i-1}^{1}$.

Since the Kalman smoother is nothing else than the Kalman filter for the
composite state $X_{0:k}$, the same induction step as in Theorem
\ref{thm:EnKF_convergence} applies for each $i$, and we have the following.

\begin{lemma}
\label{lem:EnKS-Uk-exact} The random elements $U_{0:k|k}^{1},\ldots
,U_{0:k|k}^{N}$ are i.i.d and the distribution of each $U_{0:k|k}^{n}$ is the
same as the smoothing distribution. In particular, $E(U_{0:k|k})=X_{0:k|k}$,
with $X_{0:k|k}$ is the least squares solution (\ref{eq:eq_ks_ls}).
\end{lemma}

\begin{theorem}
\label{thm:EnKS_convergence} For each time step $i=0,\ldots,k$, the random
matrix
\[
\left[
\begin{array}
[c]{c}%
X_{0:i|i}^{1},\ldots,X_{0:i|i}^{N}\\
U_{0:i|i}^{1},\ldots,U_{0:i|i}^{N}%
\end{array}
\right]  .
\]
has exchangeable columns, and $X_{0:i|i}^{1}\rightarrow U_{0:i|i}^{1}$,
$\bar{X}_{0:i|i}^N \rightarrow E(U_{0:i|i}^{1}),\text{ and }{P}_{0:i|i}%
^{N}\rightarrow{P}_{0:i|i}\text{ in }L^{p}$ as $N\rightarrow\infty$, for all
$1\leq p<\infty$.
\end{theorem}

\section{ Variational data assimilation and 4DVAR}

\label{sec:vda-4DVAR}

\subsection{4DVAR as an optimization problem}

\label{sec:4DVAR}We estimate the compound state $X_{0:k}$ of the stochastic
system (\ref{eq:filter_background_eq})--(\ref{eq:filter_observation_eq}),
conditioned on the observations $y_{1},\ldots,y_{k}$, by the maximum posterior
probability density,
\[
\mathbb{P}\left(  x_{0:k}|y_{1:k}\right)  \propto e^{-\frac{1}{2}\left(
\left\vert x_{0}-x_{\mathrm{b}}\right\vert _{B^{-1}}^{2}+\sum_{i=1}%
^{k}\left\vert x_{i}-\mathcal{M}_{i}(x_{i-1})-\mu_{i}\right\vert _{Q_{i}^{-1}%
}^{2}+\sum_{i=1}^{k}\left\vert y_{i}-\mathcal{H}_{i}(x_{i})\right\vert
_{R_{i}^{-1}}^{2}\right)  }\rightarrow\max_{x_{0:k}},
\]
which is the same as solving the nonlinear least squares problem for the
composite state $x_{0:k}$,%
\begin{equation}
\left\vert x_{0}-x_{\mathrm{b}}\right\vert _{B^{-1}}^{2}+\sum_{i=1}%
^{k}\left\vert x_{i}-\mathcal{M}_{i}(x_{i-1})-\mu_{i}\right\vert _{Q_{i}^{-1}%
}^{2}+\sum_{i=1}^{k}\left\vert y_{i}-\mathcal{H}_{i}(x_{i})\right\vert
_{R_{i}^{-1}}^{2}\rightarrow\min_{x_{0:k}}. \label{eq:nonlinear-ls}%
\end{equation}

Numerical solution of the nonlinear least squares problem
(\ref{eq:nonlinear-ls}) is the essence of weak-constraint 4-dimensional
variational data assimilation (4DVAR) \cite{Fisher-2005-EKS,Tremolet-2007-ME4}.

\subsection{The Levenberg-Marquardt method and incremental 4DVAR}

\label{sec:incremental-4DVAR}Consider an approximate solution $x_{0:k}^{j-1}$
of the nonlinear least squares problem (\ref{eq:nonlinear-ls}). We seek a
better approximation $x_{0:k}^{j}$. Linearizing the model and the observation
operators at $x_{0:k}^{j-1}$ by their tangent operators and adding a penalty
term to control the size of the increment $x_{0:k}^{j}-x_{0:k}^{j-1}$, yields
the linear least squares problem for $x_{0:k}^{j}$ in the Levenberg-Marquardt
(LM) method \cite{Levenberg-1944-MSC,Marquardt-1963-ALE} for the solution of
the nonlinear least squares (\ref{eq:nonlinear-ls}).

\begin{algorithm}
[LM method]\label{alg:LM_D} Given $x_{0:k}^{0}$ and $\gamma\geq0$, compute the
iterations $x_{0:k}^{j}$ for $j=1,2,\ldots$, as the solutions of the least
squares problem linearized at $x_{0:k}^{j-1}$,
\begin{align}
x_{0:k}^{j}  &  =%
\mathop{\mathrm{argmin}}%
_{x_{0:k}}\left\vert x_{0}-x_{b}\right\vert _{B^{-1}}^{2}+\sum_{i=1}%
^{k}\left\vert x_{i}-\mathcal{M}_{i}\left(  x_{i-1}^{j-1}\right)
-\mathcal{M}_{i}^{\prime}\left(  x_{i-1}^{j-1}\right)  \left(  x_{i-1}%
-x_{i-1}^{j-1}\right)  -\mu_{i}\right\vert _{Q_{i}^{-1}}^{2}\nonumber\\
&  +\sum_{i=1}^{k}\left\vert y_{i}-\mathcal{H}_{i}\left(  x_{i}^{j-1}\right)
-\mathcal{H}_{i}^{\prime}\left(  x_{i}^{j-1}\right)  \left(  x_{i}-x_{i}%
^{j-1}\right)  \right\vert _{R_{i}^{-1}}^{2}+\sum_{i=1}^{k}\gamma\left\vert
x_{i}-x_{i}^{j-1}\right\vert ^{2}. \label{eq:tangent-ls}%
\end{align}

\end{algorithm}

For $\gamma=0$, (\ref{eq:tangent-ls}) becomes the Gauss-Newton method, which
can converge at a rate close to quadratic, but convergence is not guaranteed
even locally. Under suitable technical assumptions, the LM method is
guaranteed to converge globally if the regularization parameter $\gamma$ is
large enough \cite{Osborne-1976-NLL, Gill-1978-ASN}, and a suitable sequence
of penalty parameters $\gamma_{j}\geq0$ changing from step to step can be
found adaptively. The LM method is a precursor of the trust-region method in
the sense that it seeks to determine when the faster Gauss-Newton method
($\gamma=0$) is applicable and when it is not and should be blended with a
slower but safer gradient descent method ($\gamma>0$). In this paper, we
consider only the case of a constant penalty parameter $\gamma>0$.

The Gauss-Newton method for the solution of nonlinear least squares is known
in atmospheric sciences as incremental 4DVAR \cite{Courtier-1994-SOI}. The use
of Levenberg-Marquardt iterations was proposed by \cite{Tshimanga-2008-LPA}.

\subsection{LM-EnKS with tangent operators}

\label{sec:EnKS-4DVAR-WD}From (\ref{eq:eq_ks_ls}), it follows that the linear
least squares problem (\ref{eq:tangent-ls}) can be interpreted as finding the
maximum posterior probability density for a linear stochastic system with all
Gaussian probability distributions. The penalty terms $\gamma |x_{i}
-x_{i}^{j-1}|^{2}$ are implemented as additional independent observations
\cite{Johns-2008-TEK} of the form%
\[
x_{i}^{j-1}=x_{i}+E_{i},\quad E_{i}\sim N\left(  0,\frac{1}{\gamma}%
I_{m}\right)  ,\quad i=1,\ldots,k.
\]

\begin{lemma}
\label{lem:LM-KS} The LM iterate $x_{0:k}^{j}$, defined by (\ref{eq:tangent-ls}%
), equals to the mean
\[
x_{0:k}^{j}=E\left(  X_{0:k|k}^{j}\right)  ,
\]
of the smoothing distribution of the stochastic system%
\begin{align}
X_{0}^{j}  &  \sim N\left(  x_{\mathrm{b}},B\right)
,\label{eq:increment-background}\\
X_{i}^{j}  &  =\mathcal{M}_{i}^{\prime}\left(  x_{i-1}^{j-1}\right)  \left(
X_{i-1}^{j}-x_{i-1}^{j-1}\right)  +\mathcal{M}_{i}\left(  x_{i-1}%
^{j-1}\right)  +\mu_{i}+V_{i}^{j},\quad V_{i}^{j}\sim N\left(  0,Q_{i}\right)
\quad i=1,\ldots,k,\label{eq:time_evol_j}\\
\tilde{y}_{i}  &  =\mathcal{\tilde{H}}_{i}\left(  x_{i}^{j-1}\right)
+\mathcal{\tilde{H}}_{i}^{^{\prime}}\left(  x_{i}^{j-1}\right)  \left(
X_{i}^{j}-x_{i}^{j-1}\right)  +\tilde{W}_{i}^{j},\quad\tilde{W}_{i}^{j}\sim
N\left(  0,\tilde{R}_{i}\right)  ,\quad i=1,\ldots,k,
\label{eq:joint_obs_gamma}%
\end{align}
where%
\begin{equation}
\tilde{y}_{i}=\left[
\begin{array}
[c]{c}%
y_{i}\\
x_i^{j-1}
\end{array}
\right]  ,\quad\mathcal{\tilde{H}}_{i}=\left[
\begin{array}
[c]{c}%
\mathcal{H}_{i}\\
I_m
\end{array}
\right]  ,\quad\tilde{R}_{i}=\left[
\begin{array}
[c]{cc}%
R_{i} & 0\\
0 & \frac{1}{\gamma}I_{m}%
\end{array}
\right]  , \label{eq:joint_obs_gamma_def}%
\end{equation}
or, equivalently%
\begin{align}
X_{0}^{j}  &  \sim N\left(  x_{\mathrm{b}},B\right)
,\label{eq:increment-background-equiv}\\
X_{i}^{j}  &  =M_{i}^{j}X_{i-1}^{j}+\tilde{\mu}_{i}^{j}+V_{i}^{j},\quad
V_{i}^{j}\sim N\left(  0,Q_{i}\right)  \quad i=1,\ldots
,k,\label{eq:time_evol_j-equiv}\\
\tilde{y}_{i}^{j}  &  =\tilde{H}_{i}^{j}X_{i}^{j}+\tilde{W}_{i}^{j}%
,\quad\tilde{W}_{i}^{j}\sim N\left(  0,\tilde{R}_{i}\right)  ,\quad
i=1,\ldots,k, \label{eq:joint_obs_gamma-equiv}%
\end{align}
where%
\begin{align*}
M_{i}^{j}  &  =\mathcal{M}_{i}^{\prime}\left(  x_{i-1}^{j-1}\right), \quad 
\tilde{\mu}_{i}^{j}    =\mathcal{M}_{i}\left(  x_{i-1}^{j-1}\right)  +\mu
_{i}-\mathcal{M}_{i}^{\prime}\left(  x_{i-1}^{j-1}\right)  x_{i-1}^{j-1},\\
\tilde{H}_{i}^{j}  &  =\mathcal{\tilde{H}}_{i}^{\prime}\left(  x_{i}%
^{j-1}\right)  ,\quad
\tilde{y}_{i}^{j}    =\tilde{y}_{i}+\mathcal{\tilde{H}}_{i}^{\prime}\left(
x_{i}^{j-1}\right)  x_{i}^{j-1}-\mathcal{\tilde{H}}_{i}\left(  x_{i}%
^{j-1}\right)  .
\end{align*}

\end{lemma}

\begin{proof}
The system (\ref{eq:increment-background})--(\ref{eq:joint_obs_gamma}) has the
same form as the original problem (\ref{eq:filter_background_eq}%
)--(\ref{eq:filter_observation_eq}) and all distributions are Gaussian, hence
Lemma \ref{lem:KS} applies.
\end{proof}

\begin{corollary}
\label{cor:LM-KS} The LM iterate $x^{j}$ is the mean found from the Kalman
smoother (\ref{eq:smoother_model})--(\ref{eq:smoother_obs_op}), applied to the
linear stochastic system (\ref{eq:increment-background}%
)--(\ref{eq:joint_obs_gamma}).
\end{corollary}

However, since the dimension of the state is generally large, we apply the
EnKS (\ref{eq:EnKS-background})--(\ref{eq:sample_mean_EnKS}) to solve
(\ref{eq:increment-background})-(\ref{eq:joint_obs_gamma}) approximately. In
each LM iteration $j=1,2,\ldots$, the linearized least squares solution
$x^{j}$ is approximated by the sample mean $\bar{X}_{0:k|k}^{j,N_{j}}$ from
the EnKS and the least squares are linearized at the previous iterate $\bar
{X}_{0:k|k}^{j-1,N_{j-1}}$ rather than at $x^{j-1}$. However, for $j=0$ this
notation is formal for the sake of consistency only. There is no ensemble for
$j=0.$

\begin{algorithm}
\label{alg:LM-EnKS-WD} Given an initial approximation $x_{0:k}^{0}$, 
 and $\gamma\geq0$. Initialize%
\[
\tilde{x}^{j}=\bar{X}_{0:k|k}^{j,N_{j}}=x_{0:k}^{0}\quad\text{for }j=0.
\]

LM loop: For $j=1,2,\ldots$. Choose an ensemble size $N_j$.

EnKS loop: For $i=0$, the ensemble $\left[  X_{0|0}^{j,n}\right]
_{n=1}^{N_{j}}$ consists of i.i.d. Gaussian random variables
\[
X_{0|0}^{j,n}\sim N\left(  \tilde{x}_{0}^{0},B\right)  .
\]

For $i=1,\ldots,k$, advance the model in time (the forecast step) by%
\begin{align}
X_{i|i-1}^{j,n}  &  =\mathcal{M}_{i}^{\prime}\left(  \bar{X}_{i-1|k}%
^{j-1,N_{j-1}}\right)  \left(  X_{i-1|i-1}^{j,n}-\bar{X}_{i-1|k}^{j-1,N_{j-1}%
}\right)  +\mathcal{M}_{i}\left(  \bar{X}_{i-1|k}^{j-1,N_{j-1}}\right)  +\mu
_{i}+V_{i}^{j,n},\label{eq:LM-EnKS-WD-adv}\\
V_{i}^{j,n}  &  \sim N\left(  0,{Q}_{i}\right)  ,\quad n=1,\ldots
,N_{j}.\nonumber
\end{align}
Incorporate the observations at time $i$ into the ensemble of composite states
$\left[  X_{0:i|i-1}^{j,n}\right]  _{n=1}^{N_{j}}$ in the same way as in the
EnKF analysis step,%
\begin{align}
X_{0:i|i}^{j,n}=  &  X_{0:i|i-1}^{j,n}+{P}_{0:i|i-1}^{j,N_{j}}\tilde{H}%
_{i}^{j\mathrm{T}}\left(  \tilde{H}_{i}^{j}{P}_{0:i|i-1}^{j,N_{j}}\tilde
{H}_{i}^{j\mathrm{T}}+\tilde{R}_{i}\right)  ^{-1}\label{eq:LM-EnKS-WD-ana}\\
&  \cdot\left(  \tilde{y}_{i}-\tilde{W}_{i}^{j,n}-\mathcal{\tilde{H}}_{i}\left(
\bar{X}_{i|k}^{j-1,N_{j-1}}\right)  -\mathcal{\tilde{H}}_{i}^{^{\prime}}\left(
\bar{X}_{i|k}^{j-1,N_{j-1}}\right)  \left(  X_{i|i-1}^{j,n}-\bar{X}%
_{i|k}^{j-1,N_{j-1}}\right)  \right)  ,\nonumber\\
\tilde{W}_{i}^{j,n}\sim &  N\left(  0,\tilde{R}_{i}\right) \nonumber
\end{align}
where ${P}_{0:i|i-1}^{j,N_{j}}$ is the sample covariance from the ensemble
$\left[  X_{0:i|i-1}^{j,n}\right]  _{n=1}^{N_{j}}$. Similarly as in
(\ref{eq:PHt})--(\ref{eq:h_n}), only the following matrix-vector products are
needed:
\begin{align*}
{P}_{0:i|i-1}^{j,N_{j}}\tilde{H}_{i}^{j\mathrm{T}}  &  =\frac{1}{N_{j}-1}%
\sum_{n=1}^{N_{j}}\left(  X_{0:i|i-1}^{j,n}-\bar{X}_{0:i|i-1}^{j,N_{j}%
}\right)  \left(  X_{i|i-1}^{j,n}-\bar{X}_{i|i-1}^{j,N_{j}}\right)^{\mathrm{T}}  \tilde
{H}_{i}^{j\mathrm{T}}\\
&  =\frac{1}{N_{j}-1}\sum_{n=1}^{N_{j}}\left(  X_{0:i|i-1}^{j,n}-\bar
{X}_{0:i|i-1}^{j,N_{j}}\right)  h_{i}^{j,n\mathrm{T}},\\
\tilde{H}_{i}^{j}{P}_{0:i|i-1}^{j,N_{j}}\tilde{H}_{i}^{j\mathrm{T}}  &
=\frac{1}{N_{j}-1}\sum_{n=1}^{N_{j}}\tilde{H}_{i}^{j}\left(  X_{i|i-1}%
^{j,n}-\bar{X}_{i|i-1}^{j,N_{j}}\right)  \left(  X_{i|i-1}^{j,n}-\bar
{X}_{i|i-1}^{j,N_{j}}\right)^{\mathrm{T}}  \tilde{H}_{i}^{j\mathrm{T}}\\
&  =\frac{1}{N_{j}-1}\sum_{n=1}^{N_{j}}h_{i}^{j,n}h_{i}^{j,n\mathrm{T}},
\end{align*}
where
\begin{equation}
h_{i}^{j,n}=\tilde{H}_{i}^{j}\left(  X_{i|i-1}^{j,n}-\bar{X}_{i|i-1}^{j,N_{j}%
}\right)  =\mathcal{\tilde{H}}_{i}^{\prime}\left(  \bar{X}_{i|k}^{j-1,N_{j-1}%
}\right)  \left(  X_{i|i-1}^{j,n}-\bar{X}_{i|i-1}^{j,N_{j}}\right)
\label{eq:LM-EnKS-WD-h}%
\end{equation}
and
\[
\bar{X}_{i|i-1}^{j,N_{j}}=\frac{1}{N_{j}}\sum_{n=1}^{N_{j}}X_{i|i-1}%
^{j,n},\quad\bar{X}_{0:i|i-1}^{j,N_{j}}=\frac{1}{N_{j}}\sum_{n=1}^{N_{j}%
}X_{0:i|i-1}^{j,n}.
\]

The next iterate is $\tilde{x}^{j}=\bar{X}_{0:k|k}^{j,N_{j}}$.
\end{algorithm}

In the rest of this section, we study the asymptotic behavior of Algorithm
\ref{alg:LM-EnKS-WD} when the ensembles sizes $N_1,\ldots N_{j}\rightarrow\infty$. We
start with an a-priori $L^{p}$ bound on the ensemble members, independent of
the ensemble size.

\begin{assumption}
\label{ass:M_H} The model and observation operators, $\mathcal{M}_{i}$, and
$\mathcal{H}_{i}$ are continuously differentiable, with at most polynomial
growth at infinity, and their Jacobians have at most polynomial growth at
infinity,
i.e. there exists $\kappa>0$ and $s\geq0$, such that $\left\vert
\mathcal{M}_{i}(x)\right\vert \leq\kappa(1+\left\vert x\right\vert ^{s})$,
$\left\vert \mathcal{M}_{i}^{\prime}(x)\right\vert \leq\kappa(1+\left\vert
x\right\vert ^{s})$, $\left\vert \mathcal{H}_{i}(x)\right\vert \leq
\kappa(1+\left\vert x\right\vert ^{s})$, and $\left\vert \mathcal{H}%
_{i}^{\prime}(x)\right\vert \leq\kappa(1+\left\vert x\right\vert ^{s})$ for
all $i$ and all $x$.
\end{assumption}

Since we are interested in the convergence with the ensemble size, we need a
notation to distinguish between $X_{0:k|k}^{j,n}$ coming from ensembles of
different sizes $N_{j}$. Thus, when we need to make such distinction, we
denote by $X_{0:k|k}^{j,n,N_{j}}$ the $n$-th ensemble member from the ensemble
$\left[  X_{0:k|k}^{j,1},\ldots,X_{0:k|k}^{j,N_{j}}\right]  $ of size $N_{j}$
in Algorithm \ref{alg:LM-EnKS-WD}, and similarly for other subscripts and superscripts.

\begin{lemma}
\label{lem:boundness_X} For any $1\leq p<\infty$, any $i=0,\ldots,k$ and any
$j=1,2,\ldots$, there exists a constant $C\left(  i,j,p\right)  $ such that in
Algorithm \ref{alg:LM-EnKS-WD},
\begin{equation}
\left\Vert {X}_{0:i|i}^{j,n,N_{j}}\right\Vert _{p}\leq C\left(  i,j,p\right)
\label{eq:boundness_X}%
\end{equation}
for all $n=1,\ldots,N_{j}$ and all $N_{j}.$
\end{lemma}

\begin{proof}
Let $p\in\lbrack1,\infty)$.
We will prove (\ref{eq:boundness_X}) by induction on the iteration number $j$. For
$j=1$, $\tilde{x}^{j-1}$ is constant, otherwise, for $j\geq2$, $\left\Vert
\tilde{x}^{j-1}\right\Vert _{p}$ is bounded independently of the ensemble
sizes by induction assumption because
\[
\tilde{x}^{j-1}=\bar{X}_{0:k|k}^{j-1,N_{j-1}}=\frac{1}{N_{j-1}}\sum
_{n=1}^{N_{j-1}}X_{0:k|k}^{j-1,n,N_{j-1}}.
\]
For a fixed $j$, we now proceed by induction on the time step $i$. For $i=0$,
$X_{0|0}^{j,n}\sim N\left(  x_{0}^{0},B\right)  $, thus $\left\Vert
X_{0|0}^{j,n}\right\Vert _{p}$ does not depend on $n$ or $N_{j}$. For
$i=1,\ldots,k$, from (\ref{eq:LM-EnKS-WD-adv}), we have%
\[
\left\Vert {X}_{i|i-1}^{j,n}\right\Vert _{p}\leq\left\Vert \mathcal{M}%
_{i}^{\prime}(\tilde{x}_{i-1}^{j-1})\right\Vert _{2p}\left(  \left\Vert
X_{i-1|i-1}^{j,n}\right\Vert _{2p}+\left\Vert \tilde{x}_{i-1}^{j-1}\right\Vert
_{2p}\right)  +\left\Vert \mathcal{M}_{i}(\tilde{x}_{i-1}^{j-1})\right\Vert
_{p}+\left\vert \mu_{i}\right\vert +\left\Vert V_{i}^{j,n}\right\Vert _{p}.
\]
From Assumption \ref{ass:M_H} and the fact that $V_{i}^{j,n}$ is normally
distributed, there exist a constant $C_{p}$ such that
\begin{align*}
\left\Vert {X}_{i|i-1}^{j,n}\right\Vert _{p}  &  \leq\kappa C_{p}\left(
1+\left\Vert \tilde{x}_{i-1}^{j-1}\right\Vert _{2ps}^{s}\right)  \left(
\left\Vert X_{i-1|i-1}^{j,n}\right\Vert _{2p}+\left\Vert \tilde{x}_{i-1}%
^{j-1}\right\Vert _{2p}\right) \\
&  +\kappa C_{p}\left(  1+\left\Vert \tilde{x}_{i-1}^{j-1}\right\Vert
_{ps}^{s}\right)  +\left\vert \mu_{i}\right\vert +C_{p}.
\end{align*}
Bounding $\tilde{x}_{i-1}^{j-1}$ by the induction assumption on $j$ and
$X_{i-1|i-1}^{j,n}$ by the induction assumption on $i$, we have that
$\left\Vert {X}_{0:i|i-1}^{j,n,N_{j}}\right\Vert _{p}$ is bounded
independently of $n$ and $N_{j}$. From equation (\ref{eq:LM-EnKS-WD-ana}), and
the fact that $\tilde{H}_{i}^{j}=\mathcal{\tilde{H}}_{i}^{\prime}\left(
\tilde{x}_{i}^{j-1}\right)  =\left[  0,\ldots,\mathcal{{H}}_{i}^{\prime
}\left(  \tilde{x}_{i}^{j-1}\right)  \right]  $ we conclude that
\begin{align*}
\left\Vert {X}_{0:i|i}^{j,n}\right\Vert _{p}\leq &  \left\Vert {X}%
_{0:i|i-1}^{j,n}\right\Vert _{p}+\left\Vert {P}_{0:i|i-1}^{j,N_{j}}%
\tilde{\mathcal{H}}_{i}^{\prime\mathrm{T}}\left(  \tilde{x}_{i}^{j-1}\right)
\left(  \tilde{\mathcal{H}}_{i}^{\prime}\left(  \tilde{x}_{i}^{j-1}\right)
{P}_{0:i|i-1}^{j,N_{j}}\tilde{\mathcal{H}}_{i}^{\prime\mathrm{T}}\left(
\tilde{x}_{i}^{j-1}\right)  +\tilde{R}_{i}\right)  ^{-1}\right\Vert _{2p}\\
&  .\left\Vert \tilde{y}_{i}-\tilde{W}_{i}^{j,n}-\mathcal{H}_{i}\left(  \tilde{x}%
_{i}^{j-1}\right)  -\tilde{\mathcal{H}}_{i}^{\prime}\left(  \tilde{x}%
_{i}^{j-1}\right)  \left(  X_{i|i-1}^{j,n}-\tilde{x}_{i}^{j-1}\right)
\right\Vert _{2p},\\
\leq &  \left\Vert {X}_{0:i|i-1}^{j,n}\right\Vert _{p}+\left\Vert
{P}_{0:i|i-1}^{j,N_{j}}\right\Vert _{8p}\left\Vert \tilde{\mathcal{H}}%
_{i}^{\prime\mathrm{T}}\left(  \tilde{x}_{i}^{j-1}\right)  \right\Vert _{8p}\\
&  \cdot\left\Vert \left(  \tilde{\mathcal{H}}_{i}^{\prime}\left(  \tilde
{x}_{i}^{j-1}\right)  {P}_{0:i|i-1}^{j,N_{j}}\tilde{\mathcal{H}}_{i}%
^{\prime\mathrm{T}}\left(  \tilde{x}_{i}^{j-1}\right)  +\tilde{R}_{i}\right)
^{-1}\right\Vert _{4p}\\
&  .\left( \left\vert \tilde{y}_{i}\right\vert +\left\Vert \tilde{W}_{i}^{j,n}\right\Vert
_{2p}+\left\Vert \mathcal{H}_{i}\left(  \tilde{x}_{i}^{j-1}\right)
\right\Vert _{2p} +\left\Vert \tilde{\mathcal{H}}_{i}^{\prime}\left(  \tilde{x}_{i}%
^{j-1}\right)  \right\Vert _{4p}\left(  \left\Vert X_{i|i-1}^{j,n}\right\Vert
_{4p}+\left\Vert \tilde{x}_{i}^{j-1}\right\Vert _{4p}\right) \right) .
\end{align*}
Since $\tilde{R}_{i}$ is positive definite and $P_{0:i|i-1}^{j,N_{j}}$ is positive
semi definite, we have%
\begin{equation}
\left\Vert \left(  \tilde{\mathcal{H}}_{i}^{\prime}\left(  \tilde{x}_{i}%
^{j-1}\right)  {P}_{0:i|i-1}^{j,N_{j}}\tilde{\mathcal{H}}_{i}^{\prime
\mathrm{T}}\left(  \tilde{x}_{i}^{j-1}\right)  +\tilde{R}_{i}\right)  ^{-1}%
\right\Vert _{4p}\leq\left| \tilde{R}_{i}^{-1}\right| . \label{eq:kg-bound}%
\end{equation}
From \cite[lemma 31]{Mandel-2011-CEK} we have
\begin{equation}
\left\Vert P_{0:i|i-1}^{j,N_{j}}\right\Vert _{8p}\leq2\left\Vert
X_{0:i|i-1}^{j,1}\right\Vert _{16p}^{2}. \label{eq:cov-bound}%
\end{equation}
From the inequalities (\ref{eq:kg-bound}) and (\ref{eq:cov-bound}), Assumption
\ref{ass:M_H}, and the fact that $\tilde{W}_{i}^{j,n}$ is normally distributed, there
exists a constant $\tilde{C}_{p}$ such that
\begin{align*}
\left\Vert {X}_{0:i|i}^{j,n}\right\Vert _{p}\leq &  \left\Vert {X}%
_{0:i|i-1}^{j,n}\right\Vert _{p}+2\left\Vert {X}_{0:i|i-1}^{j,1}\right\Vert
_{16p}^{2}\kappa\tilde{C}_{p}\left(  1+\left\Vert \tilde{x}_{i}^{j-1}%
\right\Vert _{8ps}^{s}\right) \\
&  \left\vert \tilde{R}_{i}^{-1}\right\vert \left\vert \tilde{y}_{i}\right\vert
+\tilde{C}_{p}+\kappa\tilde{C}_{p}\left(1+\left\Vert \tilde{x}_{i}^{j-1}\right\Vert
_{2ps}^{s}\right)\\
&  +\kappa\tilde{C}_p\left(  1+\left\Vert \tilde{x}_{i}^{j-1}\right\Vert
_{4ps}^{s}\right)  \left(  \left\Vert X_{i|i-1}^{j,1}\right\Vert
_{4p}+\left\Vert \tilde{x}_{i}^{j-1}\right\Vert _{4p}\right)
\end{align*}
Bounding $\tilde{x}_{i}^{j-1}$ by the induction assumption on $j$ and
$X_{i-1|i-1}^{j,n}$ by the induction assumption on $i$, we obtain that
$\left\Vert {X}_{0:i|i}^{j,n,N_{j}}\right\Vert _{p}$ is bounded independently
of $n$ and $N_{j}$.
\end{proof}

At each iteration $j$ of Algorithm \ref{alg:LM-EnKS-WD}, we define for
theoretical purposes a reference ensemble $\left[  U_{0:k|k}^{1},\ldots,U_{0:k|k}^{N_{j}%
}\right]  $, similarly as in Section \ref{sec:EnKS_convergence}, and with the
derivatives taken at the mean, rather than sample mean as in Algorithm
\ref{alg:LM-EnKS-WD}: For $j=0$, all $U_{i|i}^{j,n}=X_{i|i}^{j,n}=x^{0}$ are
constants. For $j=1,2,\ldots,$ $U_{0|0}^{j,n}=X_{0|0}^{j,n}$ for $i=0,$ and for
$i=1,\ldots,k$,
\begin{align}
U_{i|i-1}^{j,n}  &  =\mathcal{M}_{i}^{\prime}\left(  E\left(  U_{i-1|i-1}%
^{j-1,1}\right)  \right)  \left(  U_{i-1|i-1}^{j,n}-x_{i-1}^{j-1}\right)
+\mathcal{M}_{i}^{\prime}\left(  E\left(  U_{i-1}^{j-1,1}\right)  \right)
+\mu_{i}+V_{i}^{j,n},\label{eq:def-Uk-adv-KS}\\
U_{0:i|i}^{j,n}  &  =
{U}_{0:i|i-1}^{j,n}+Q_{0:i}^{j}\tilde{\mathcal{H}}_{i}^{\prime}\left(E\left(U_{i|k}^{j-1,1}\right)\right)^{\mathrm{T}}
\left( \mathcal{H}_{i}^{\prime}\left(E\left(U_{i|k}^{j-1,1}\right)\right)Q_{i}%
^{j}\mathcal{H}_{i}^{\prime}\left(  E\left(  U_{i|k}^{j-1,1}\right)  \right)
^{\mathrm{T}}+\tilde{R}_{i}\right)  ^{-1} \nonumber \\
&
\left(\hat{y}_{i}+\tilde{\mathcal{H}}_{i}^{\prime}\left(  E\left(
{U}_{i|k}^{j-1,1}\right)  \right)\left(  E\left(  {U}_{i|k}^{j-1,1}\right) - {U}_{i|i-1}^{j,1} \right)
-\mathcal{H}_{i}\left(  E\left(  {U}_{i|i}^{j-1,1}\right)  \right)- \tilde{W}_{i}^{j,1} \right)
 \label{eq:def-Uk-ana-KS}%
\end{align}
where
 $Q_{0:i}^{j}=\operatorname{Cov}\left(  U_{0:i|i-1}^{j-1,1}\right)  $ and $\hat{y}_{i} = \left[
\begin{array}
[c]{c}%
y_{i}\\
E\left(  U_{i|k}^{j-1,1}\right)
\end{array}
\right]$. We
now show that the mean of the reference ensemble members is the solution of
the linearized least squares (\ref{eq:tangent-ls}), and thus the next LM iterate:

\begin{lemma}
\label{lem:U_eq_min_r} $E(U_{0:k|k}^{j,1})=x^{j}$, where $x^{j}$ is the $j$-th
iterate generated by the algorithm (\ref{alg:LM_D}).
\end{lemma}

\begin{proof}
The proof proceeds by induction on the iteration number $j$. For $j=0$ we have
$U^{0,n}=x^{0}$ for all $n$ by definition, thus $E(U^{0,1})=x^{0}$. Let
$j\geq1$. From the induction assumption, the stochastic system
(\ref{eq:def-Uk-adv-KS})--(\ref{eq:def-Uk-ana-KS}) is linearized about the
previous LM iterate $x^{j-1}$. By Lemma \ref{lem:EnKS-Uk-exact}, $\left[
U_{0:k}^{j,n}\right]  _{n=1}^{N_{j}}$ are i.i.d. with the smoothing
distribution, whose mean is the solution of the linearlized least squares
(\ref{eq:tangent-ls}).
\end{proof}

\begin{lemma}
 For all iterations $j$ and all times $i=0,\ldots,k$, the columns of the random
matrix
\begin{equation}
\left[  {X}_{0:i|i}^{j};U_{0:i|i}^{j}\right]  =\left[
\begin{array}
[c]{c}%
{X}_{0:i|i}^{j,1},\ldots,{X}_{0:i|i}^{j,N_{j}}\\
U_{0:i|i}^{j,1},\ldots,U_{0:i|i}^{j,N}%
\end{array}
\right]  \label{eq:exchange_j}%
\end{equation}
are exchangeable. 
\end{lemma}

\begin{proof}
We recall that for all $j\geq1$, $\tilde{x}^{j}=\bar{X}_{0:k|k}^{j,N_{j}}$. In  this proof we omit the subscripts of $N_j$ and $N_{j-1}$. 
 We use induction on the LM iteration number $j$. For $j=0$ we have that for all
$i=0,\ldots,k,$ $X_{0:i|i}^{0,n}=x_{0:i}^{0}=U_{0:i|i}^{0,n}$, and
$U_{0:i|i}^{0,n}$ are i.i.d., therefore $\left[  {X}_{0:i|i}^{0};U_{0:i|i}%
^{0}\right]  $ are exchangeable. For $j\geq1$, we use the induction on the
time index $i$. For $i=0$, $\left[  U_{0|0}^{j,n}\right]  _{n=1}^{N}$ are
i.i.d, and $X_{0|0}^{j,n}=U_{0|0}^{j,n}$, therefore $\left[  X_{0|0}%
^{j};U_{0|0}^{j}\right]  $ are exchangeable. For $i=1,\ldots,k$, consider
first the forecast step,
\begin{align*}
\left[
\begin{array}
[c]{c}%
X_{i|i-1}^{j,n}\\
U_{i|i-1}^{j,n}%
\end{array}
\right]   &  =\left[
\begin{array}
[c]{ll}%
\mathcal{M}_{i}^{\prime}\left(  \bar{X}_{i-1|k}^{j-1,N}\right)  & 0\\
0 & \mathcal{M}_{i}^{\prime}\left(  E\left(  U_{i-1|k}^{j-1,1}\right)
\right)
\end{array}
\right]  \left[
\begin{array}
[c]{c}%
X_{i-1|i-1}^{j,n}\\
U_{i-1|i-1}^{j,n}%
\end{array}
\right] \\
&  +\left[
\begin{array}
[c]{c}%
\mathcal{M}_{i}\left(  \bar{X}_{i-1|k}^{j-1,N}\right)  -\mathcal{M}%
_{i}^{\prime}\left(\bar{X}_{i-1|k}^{j-1,N}\right)\bar{X}_{i-1|k}^{j-1,N}+\mu_{i}\\
\mathcal{M}_{i}\left(E\left(U_{i-1|k}^{j-1,1}\right)\right)-\mathcal{M}_{i}^{\prime}\left(E\left(U_{i-1|k}%
^{j-1,1}\right)\right)E\left(U_{i-1|k}^{j-1,1}\right)+\mu_{i}%
\end{array}
\right]  +\left[
\begin{array}
[c]{c}%
V_{i}^{j,n}\\
V_{i}^{j,n}%
\end{array}
\right] \\
&  =F^{i}\left(  \bar{X}_{i-1|k}^{j-1,N},\left[
\begin{array}
[c]{c}%
X_{i-1|i-1}^{j,n}\\
U_{i-1|i-1}^{j,n}%
\end{array}
\right]  ,\left[
\begin{array}
[c]{c}%
V_{i}^{j,n}\\
V_{i}^{j,n}%
\end{array}
\right]  \right)  ,
\end{align*}
where $F^{i}$ is a measurable function. The ensemble sample mean $\bar
{X}_{i-1|k}^{j-1,N}$ is invariant to a permutation of ensemble members, and
$V_{i}^{j}=\left[  V_{i}^{j,1},\ldots,V_{i}^{j,N}\right]  $ is exchangeable
because its members are i.i.d. From the induction assumption on $i$ we have
that $\left[  X_{i-1|i-1}^{j};U_{i-1|i-1}^{j}\right]  $ is exchangeable, and
it is also independent from $\left[
\begin{array}
[c]{c}%
V_{i}^{j}\\
V_{i}^{j}%
\end{array}
\right]  $, therefore, using Lemma~\ref{lem:exchangeability_F}, $\left[
\begin{array}
[c]{c}%
X_{i|i-1}^{j}\\
U_{i|i-1}^{j}%
\end{array}
\right]  $ is exchangeable. The analysis step also preserves exchageability:%
\begin{align*}
\left[
\begin{array}
[c]{c}%
X_{0:i|i}^{j,n}\\
U_{0:i|i}^{j,n}%
\end{array}
\right]   &  =\left[
\begin{array}
[c]{c}%
X_{0:i|i-1}^{j,n}\\
U_{0:i|i-1}^{j,n}%
\end{array}
\right]  +\left[
\begin{array}
[c]{ll}%
K_{i}^{N} & 0\\
0 & K_{i}%
\end{array}
\right] \\
&  \left(  \left[
\begin{array}
[c]{c}%
\tilde{y}_{i}-\mathcal{H}_{i}^{\prime}\left(  \bar{X}_{i|k}^{j-1,N}\right)  \bar
{X}_{i|i}^{j-1,N}-\mathcal{H}_{i}\left(  \bar{X}_{i|k}^{j-1,N}\right)
-\tilde{W}_{i}^{j,n}\\
\hat{y}_{i}-\mathcal{H}_{i}^{\prime}\left(  E\left(  {{U}}_{i|k}^{j-1,1}\right)
\right)  E\left(  {U}_{i|k}^{j-1,1}\right)  -\mathcal{H}_{i}\left(  E\left(
{U}_{i|k}^{j-1,1}\right)  \right)  -\tilde{W}_{i}^{j,n}%
\end{array}
\right]  \right. \\
&  -\left.  \left[
\begin{array}
[c]{ll}%
\mathcal{H}_{i}^{\prime}\left(  \bar{X}_{i|k}^{j-1,N}\right)  & 0\\
0 & \mathcal{H}_{i}^{\prime}\left(  E\left(  U_{i|k}^{j-1,1}\right)  \right)
\end{array}
\right]  \left[
\begin{array}
[c]{c}%
X_{i|i-1}^{j,n}\\
U_{i|i-1}^{j,n}%
\end{array}
\right]  \right) \\
&  =F^{i}\left(  \bar{X}_{i|k}^{j-1,N},P_{0:i|i-1}^{j,N},\left[
\begin{array}
[c]{c}%
X_{i|i-1}^{j,n}\\
U_{i|i-1}^{j,n}%
\end{array}
\right]  ,\left[
\begin{array}
[c]{c}%
\tilde{W}_{i}^{j,n}\\
\tilde{W}_{i}^{j,n}%
\end{array}
\right]  \right)
\end{align*}
because the Kalman gain matrices are functions of the ensemble members through
$\bar{X}_{i|k}^{j-1,N}$ and $P_{0:i|i-1}^{j,N}$ only,
\begin{align*}
K_{i}^{N}  &  =\left[
\begin{array}
[c]{c}%
P_{0|i-1}^{j,N}\mathcal{H}_{i}^{\prime}(\bar{X}_{i|k}^{j-1,N})^{\mathrm{T}%
}\\
\vdots\\
P_{i|i-1}^{j,N}\mathcal{H}_{i}^{\prime}\left(  \bar{X}_{i|k}^{j-1,N}\right)
^{\mathrm{T}}%
\end{array}
\right]  \left(  \mathcal{H}_{i}^{\prime}\left(  \bar{X}_{i|k}^{j-1,Nj}%
\right)  P_{i|i-1}^{j,N}\mathcal{H}_{i}^{\prime}\left(  \bar{X}%
_{i|k}^{j-1,N}\right)  ^{\mathrm{T}}+\tilde{R}_{i}\right)  ^{-1},\\
K_{i}  &  =\left[
\begin{array}
[c]{c}%
Q_{0}^{j}\mathcal{H}_{i}^{\prime}\left(  E\left(  U_{i|k}^{j-1,1}\right)
\right)  ^{\mathrm{T}}\\
\vdots\\
Q_{i}^{j}\mathcal{H}_{i}^{\prime}\left(  E\left( U_{i|k}^{j-1,1}\right)
\right)  ^{\mathrm{T}}%
\end{array}
\right]  \left(  \mathcal{H}_{i}^{\prime}\left(E\left(U_{i|k}^{j-1,1}\right)\right)Q_{i}
^{j}\mathcal{H}_{i}^{\prime}\left(  E\left(  U_{i|k}^{j-1,1}\right)  \right)
^{\mathrm{T}}+\tilde{R}_{i}\right)  ^{-1}%
\end{align*}
and the ensemble sample mean $\bar{X}_{i|k}^{j-1,N}$, and the ensemble sample
covariance $P_{0:i|i-1}^{j,N}$ are invariant to a permutation of ensemble
members, $\tilde{W}_{i}^{j}=[\tilde{W}_{i}^{j,1},\ldots,\tilde{W}_{i}^{j,N}]$ is exchangeable because,
its members are i.i.d. and they are independent from $\left[  X_{i|i-1}%
^{j,n};U_{i|i-1}^{j,n}\right]  $ is exchangeable. Therefore, using again
Lemma~\ref{lem:exchangeability_F}, $\left[  X_{0:i|i}^{j,n};U_{0:i|i}%
^{j,n}\right]  $ are exchangeable.
\end{proof}

\begin{theorem}
\label{thm:conv-N} For all $j$, and $n=1,\ldots N_{j}$, ${X}_{0:k|k}%
^{j,n,N_{j}}\rightarrow U_{0:k|k}^{j,n}$ and $\bar{X}_{0:k|k}^{j,N_{j}%
}\rightarrow E\left(  U_{0:k|k}^{j,1} \right)  $ as $\min\{N_{1},\ldots
,N_{j}\}\rightarrow\infty\text{, in all }L^{p}$, $1\leq p<\infty$.

\end{theorem}

\begin{proof}
We will prove that for all $j\geq0$ and all $1\leq i\leq k$, ${X}%
_{0:i|i}^{j,1,N_{j}}\rightarrow U_{0:i|i}^{j,1}$ as\newline$\min\{N_{1}%
,\ldots,N_{j}\}\rightarrow\infty\text{, in all }L^{p}$, $1\leq p<\infty$, and
the convergence of the mean follows. Since $\left[  X_{0:i|i}^{j,n}%
;U_{0:i|i}^{j,n}\right]  _{n=1}^{N_{j}}$ are exchangeable, we only need to
consider the convergence of $X_{0:i|i}^{j,1,N_{j}}\rightarrow U_{0:i|i}^{j,1}%
$. We use induction on the LM iteration number $j$. For $j=0$, we have
$X_{0:i|i}^{0,1,N_{0}}=x_{0:i}^{0}=U_{0:i|i}^{0,1}$. For $j\geq1$, we use
induction on time index $i$. For $i=0$, $X_{0|0}^{j,1,N_{j}}=U_{0|0}^{j,1}$.
For $i=1,\ldots,k$, from induction assumption on $j$ and $i$, we have $\bar
{X}_{i-1|k}^{j-1,N_{j}}\rightarrow E(U_{i-1|k}^{j-1,1})$ and $X_{i-1|k}%
^{j,1,N_{j}}\rightarrow U_{i-1|k}^{j,1}$ in all $L^{p}$, $1\leq p<\infty$, as
$\min\{N_{1},\ldots,N_{j}\}\rightarrow\infty$. Convergence in $L^{p}$ implies
convergence in probability, and by the continuous mapping theorem,
\begin{align*}
{X}_{i|i-1}^{j,1,N_{j}}=  &  \mathcal{M}_{i}^{\prime}\left(  \bar{X}%
_{i-1|k}^{j-1,N_{j-1}}\right)  X_{i-1|i-1}^{j,1,N_{j}}+\mathcal{M}_{i}\left(
\bar{X}_{i-1|k}^{j-1,N_{j-1}}\right) \\
&  -\mathcal{M}_{i}^{\prime}\left(  \bar{X}_{i-1|k}^{j-1,N_{j-1}}\right)
\bar{X}_{i-1|k}^{j-1,N_{j-1}}+\mu_{i}+V_{i}^{j,1}\\%
\xrightarrow{\mathrm{P}}%
&  \mathcal{M}_{i}^{\prime}\left(  E\left(  U_{i-1|k}^{j-1,1}\right)  \right)
U_{i-1|i-1}^{j,1}+\mathcal{M}_{i}\left(  E\left(  U_{i-1|k}^{j-1,1}\right)
\right) \\
&  -\mathcal{M}_{i}^{\prime}\left(  E\left(  U_{i-1|k}^{j-1,1}\right)
\right)  E\left(  U_{i-1|k}^{j-1,1}\right)  +\mu_{i}+V_{i}^{j,1}\\
=  &  {U}_{i|i-1}^{j,1}%
\end{align*}
as $\min\{N_{1},\ldots,N_{j}\}\rightarrow\infty$. From Lemma
\ref{lem:boundness_X}, the sequence $\left\{  {X_{0:i|i-1}^{j,1,N_{j}}%
}\right\}  _{N_{j}=1}^{\infty}$ is bounded in all $L^{p}$, $1\leq p<\infty$,
therefore by using the uniform integrability theorem we leverage the
convergence in probability to convergence in all $L^{p}$, hence $X_{0:i|i-1}%
^{j,1,N_{j}}\rightarrow U_{0:i|i-1}^{j,1}$ and $\bar{X}_{0:i|i-1}^{j,N_{j}%
}\rightarrow E\left(  U_{0:i|i-1}^{j,1}\right)  $ in all $L^{p}$. From
\cite{Mandel-2011-CEK} we have $P_{0:i|i-1}^{j,N_{j}}%
\xrightarrow{\mathrm{P}}%
Q_{0:i}^{j}$, then, from the continuous mapping theorem, $K_{i}^{N_{j}%
}\xrightarrow{\mathrm{P}}K_{i}$. From the fact that convergence in $L^{p}$ implies convergence in
probability, and using the continuous mapping theorem again, we conclude that
\begin{align*}
{X}_{0:i|i}^{j,1,N_{j}}  &  ={X}_{0:i|i-1}^{j,1,N_{j}}+K_{i}^{N_{j}}%
\left(\tilde{y}_{i}+\tilde{\mathcal{H}}_{i}^{\prime}\left(  \bar{X}_{i|k}^{j-1,N_{j-1}}\right)
\bar{X}_{i|k}^{j-1,N_{j-1}}-\tilde{\mathcal{H}}_{i}\left(  \bar{X}_{i|k}^{j-1,N_{j-1}%
}\right) \right.\\
&  \quad-\left.\tilde{W}_{i}^{j,1}-\tilde{\mathcal{H}}_{i}^{\prime}\left(  \bar{X}_{i|k}%
^{j-1,N_{j-1}}\right)  X_{i|i-1}^{j,1,N_{j}}\right)\\
&
\xrightarrow{\mathrm{P}}%
{U}_{0:i|i-1}^{j,1}+K_{i}\left(\hat{y}_{i}+\tilde{\mathcal{H}}_{i}^{\prime}\left(  E\left(
{U}_{i|k}^{j-1,1}\right)  \right)  E\left(  {U}_{i|k}^{j-1,1}\right)
-\tilde{\mathcal{H}}_{i}\left(  E\left(  {U}_{i|i}^{j-1,1}\right)  \right) \right. \\
&  \quad-\left. \tilde{W}_{i}^{j,1}-\tilde{\mathcal{H}}_{i}^{\prime}\left(  E\left(  U_{i|k}%
^{j-1,1}\right)  \right)  {U}_{i|i-1}^{j,1}\right)\\
&  ={U}_{0:i|i}^{j,1},
\end{align*}
as $\min\left\{  N_{1},\ldots,N_{j}\right\}  \rightarrow\infty$. Then we
leverage the last convergence to the convergence in $L^{p}$ using Lemma~\ref{lem:boundness_X} and the uniform
integrability again.
\end{proof}

\subsection{EnKS-4DVAR}

\label{sec:EnKS-4DVAR-fd}

To avoid computing with the tangent matrices $\mathcal{M}_{i}^{\prime}\left(
x_{i-1}^{j-1}\right)  $ and \newline$\mathcal{H}_{i}^{\prime}\left(  x_{i}%
^{j-1}\right)  $, we take advantage of the fact that they occur in the EnKS
only in matrix-vector products, and approximate the matrix-vector
multiplications in Algorithm \ref{alg:LM-EnKS-WD} by finite differences with a
small step size $\tau>0$, centered at the previous iterate. Thus, we use the
approximations of the form
\begin{equation}
f^{\prime}\left(  x\right)  y\approx\frac{f\left(  x+\tau y\right)
-f\left(  y\right)  }{\tau} \label{eq:fd}%
\end{equation}
in (\ref{eq:LM-EnKS-WD-adv}), (\ref{eq:LM-EnKS-WD-ana}), and
(\ref{eq:LM-EnKS-WD-h}). Denote by an additional superscript $\tau$ the
quantities computed in the resulting algorithm. This is the EnKS-4DVAR method
originally proposed in \cite{Mandel-2013-4EK}.

\begin{algorithm}
[\textbf{\textsc{{EnKS-4DVAR}}}]\label{alg:LM_EnKS2} Given an initial
approximation $x_{0:k}^{0}$,
 $\gamma>0$, and $\tau>0$. Initialize%
\[
\bar{X}_{0:k|k}^{j,N_{j},\tau}=x_{0:k}^{0}\quad\text{for } j=0.
\]

LM loop: For $j=1,2,\ldots$. Choose $N_j$ the same as in Algorithm~\ref{alg:LM-EnKS-WD}.

EnKS loop: For $i=0$, the ensemble $\left[  X_{0|0}^{j,n,\tau}\right]
_{n=1}^{N_{j}}$ consists of i.i.d. Gaussian random variables
\[
X_{0|0}^{j,n,\tau}\sim N\left(  \tilde{x}_{0}^{0},B\right)  .
\]

For $i=1,\ldots,k$, advance the model in time (the forecast step) by%
\begin{align}
X_{i|i-1}^{j,n,\tau}  &  =\frac{\mathcal{M}_{i}\left(  \bar{X}_{i-1|k}%
^{j-1,N_{j-1},\tau}+\tau\left(  X_{i-1|i-1}^{j,n,\tau}-\bar{X}_{i-1|k}%
^{j-1,N_{j-1},\tau}\right)  \right)  -\mathcal{M}_{i}\left(  \bar{X}%
_{i-1|k}^{j-1,N_{j-1},\tau}\right)  }{\tau}  ,\label{eq:LM-EnKS2-adv}\\
&  +\mathcal{M}_{i}\left(  \bar
{X}_{i-1|k}^{j-1,N_{j-1},\tau}\right) +\mu_{i}+V_{i}^{j,n}\quad V_{i}^{j,n}\sim N\left(  0,{Q}_{i}\right)  ,\quad
n=1,\ldots,N_{j}.\nonumber
\end{align}
Incorporate the observations at time $i$ into the ensemble of composite states
$\left[  X_{0:i|i-1}^{j,n,\tau}\right]  _{n=1}^{N_{j}}$ by the analysis step
\begin{align}
X_{0:i|i}^{j,n,\tau}=  &  X_{0:i|i-1}^{j,n,\tau}+{P}_{0:i|i-1}^{j,N_{j},\tau
}\tilde{H}_{i}^{j,\tau\mathrm{T}}\left(  \tilde{H}_{i}^{j,\tau}{P}%
_{0:i|i-1}^{j,N_{j},\tau}\tilde{H}_{i}^{j,\tau\mathrm{T}}+\tilde{R}_{i}\right)
^{-1}\label{eq:LM-EnKS2-ana}\\
&  \cdot\Bigg(\tilde{y}_{i}-\tilde{W}_{i}^{j,n}-\mathcal{\tilde{H}}_{i}\left(  \bar
{X}_{i|k}^{j-1,N_{j-1},\tau}\right) \nonumber\\
&  \quad-\frac{\mathcal{\tilde{H}}_{i}\left(  \bar{X}_{i|k}^{j-1,N_{j-1},\tau
}+\tau\left(  X_{i|i-1}^{j,n,\tau}-\bar{X}_{i|k}^{j-1,N_{j-1},\tau}\right)
\right)  -\mathcal{\tilde{H}}_{i}\left(  \bar{X}_{i|k}^{j-1,N_{j-1},\tau
}\right)  }{\tau}\Bigg),\nonumber\\
&  \tilde{W}_{i}^{j,n}\sim N\left(  0,\tilde{R}_{i}\right) \nonumber
\end{align}
where $P_{0:i|i-1}^{j,N_{j},\tau}$ is the sample covariance from the ensemble
$\left[  X_{0:i|i-1}^{j,n,\tau}\right]  _{n=1}^{N_{j}}$. Similarly as in
(\ref{eq:PHt})--(\ref{eq:h_n}), only the following matrix-vector products are
needed:
\begin{align}
{P}_{0:i|i-1}^{j,N_{j},\tau}\tilde{H}_{i}^{j,\tau\mathrm{T}}  &  =\frac
{1}{N_{j}-1}\sum_{n=1}^{N_{j}}\left(  X_{0:i|i-1}^{j,n,\tau}-\bar{X}%
_{0:i|i-1}^{j,N_{j},\tau}\right)  \left(  X_{i|i-1}^{j,n,\tau}-\bar{X}%
_{i|i-1}^{j,N_{j},\tau}\right)^{\mathrm{T}}  \tilde{H}_{i}^{j,\tau\mathrm{T}}%
,\label{eq:LM-EnKS2-PHt}\\
&  =\frac{1}{N_{j}-1}\sum_{n=1}^{N_{j}}\left(  X_{0:i|i-1}^{j,n,\tau}-\bar
{X}_{0:i|i-1}^{j,N_{j},\tau}\right)  h_{i}^{j,n,\tau\mathrm{T}}\nonumber\\
\tilde{H}_{i}^{j}{P}_{0:i|i-1}^{j,N_{j},\tau}\tilde{H}_{i}^{j,\tau\mathrm{T}}
&  =\frac{1}{N_{j}-1}\sum_{n=1}^{N_{j}}\tilde{H}_{i}^{j,\tau}\left(
X_{i|i-1}^{j,n,\tau}-\bar{X}_{i|i-1}^{j,N_{j},\tau}\right)  \left(
X_{i|i-1}^{j,n,\tau}-\bar{X}_{i|i-1}^{j,N_{j},\tau}\right)^{\mathrm{T}}  \tilde{H}%
_{i}^{j,\tau\mathrm{T}}\label{eq:LM-EnKS2-HPHt}\\
&  =\frac{1}{N_{j}-1}\sum_{n=1}^{N_{j}}h_{i}^{j,n,\tau}h_{i}^{j,n,\tau
\mathrm{T}},\nonumber
\end{align}
where
\small{
\begin{equation}
h_{i}^{j,n,\tau}=\tilde{H}_{i}^{j,\tau}\left(  X_{i|i-1}^{j,n,\tau}-\bar
{X}_{i|i-1}^{j,N_{j},\tau}\right)  =\frac{\mathcal{\tilde{H}}_{i}\left(
\tau\left(  X_{i|i-1}^{j,n,\tau}-\bar{X}_{i|i-1}^{j,N_{j},\tau}\right)
+\bar{X}_{i|k}^{j-1,N_{j-1},\tau}\right)  -\mathcal{\tilde{H}}_{i}\left(
\bar{X}_{i|k}^{j-1,N_{j-1},\tau}\right)  }{\tau} \label{eq:LM-EnKS2-hn}%
\end{equation}
}
and
\[
\bar{X}_{i|i-1}^{j,N_{j},\tau}=\frac{1}{N_{j}}\sum_{n=1}^{N_{j}}%
X_{i|i-1}^{j,n,\tau},\quad\bar{X}_{0:i|i}^{j,N_{j},\tau}=\frac{1}{N_{j}}%
\sum_{n=1}^{N_{j}}X_{0:i|i}^{j,n,\tau}.
\]

The next LM iterate is $\tilde{x}^{j,\tau}=\bar{X}_{0:k|k}^{j,N_{j},\tau}$.
\end{algorithm}

We now summarize the differences between the previous three
algorithms.\ Algorithm \ref{alg:LM_D} solves the linearized problem in each
iteration exactly, while Algorithm \ref{alg:LM-EnKS-WD} approximates the
solution of the linearized problem by EnKS, and Algorithm \ref{alg:LM_EnKS2}
approximates also the linearized problem itself by finite differences.

We show that when the finite difference parameter $\tau\rightarrow0$, the
iterations of Algorithm \ref{alg:LM_EnKS2} converge to their corresponding
iterations of Algorithm \ref{alg:LM-EnKS-WD} in probability.
The following lemma is the cornerstone of the analysis of the finite
differences here.

\begin{lemma}
\label{lem:fd-conv} Let $\left(  X_{\tau}\right)  $ and $\left(  Y_{\tau
}\right)  $ be random vectors such that $X_{\tau}%
\xrightarrow{\mathrm{P}}%
X$ and $Y_{\tau}%
\xrightarrow{\mathrm{P}}%
Y$ as $\tau\rightarrow0$, $\tau>0$, and $f$ be twice continuously
differentiable with the matrix of second order derivatives $f^{\prime\prime}$
bounded. Then,
\[
\frac{f(X_{\tau}+\tau Y_{\tau})-f(X_{\tau})}{\tau}%
\xrightarrow{\mathrm{P}}%
f^{^{\prime}}(X)Y\text{ as }\tau\rightarrow0\text{, }\tau>0.
\]

\end{lemma}

\begin{proof}
From Taylor expansion, for any $x$, $y$, and $t$,%
\begin{equation}
\left\vert \frac{f(x+ty)-f\left(  x\right)  }{t}-f^{^{\prime}}(x)y\right\vert
\leq Mt\left\vert y\right\vert ^{2}, \label{eq:fd-estimate}%
\end{equation}
where $M=\frac{1}{2}\sup_{\xi}\left\vert f^{\prime\prime}\left(  \xi\right)
\right\vert $ in the matrix norm induced by the vector norm $\left\vert
\cdot\right\vert $. Let $\varepsilon>0,$ $\tilde{\varepsilon}>0$. Since
$Y_{\tau}%
\xrightarrow{\mathrm{P}}%
Y$, $\left\{  Y_{\tau}\right\}  $ is uniformly tight, that is, there exists
$K$ such that $\mathbb{P}\left[  \left\vert Y_{\tau}\right\vert \leq K\right]
\geq1-\tilde{\varepsilon}$ for all $\tau>0$. Choose $\tau_{1}=\frac
{\varepsilon}{MK^{2}}>0$. Using (\ref{eq:fd-estimate}), it follows that for
all $0<\tau<\tau_{1}$,
\begin{equation}
\mathbb{P}\left[  \left\vert \frac{f(X_{\tau}+\tau Y_{\tau})-f(X_{\tau})}%
{\tau}-f^{^{\prime}}(X_{\tau})Y_{\tau}\right\vert \leq\varepsilon\right]
\geq1-\tilde{\varepsilon}. \label{eq:fd-prob-est}%
\end{equation}
Since the mapping $\left(  x,y\right)  \mapsto f^{\prime}\left(  x\right)  y$
is continuous and $\left(  X_{\tau},Y_{\tau}\right)  \rightarrow\left(
X,Y\right)  $ in probability, it follows from the continuous mapping theorem
that $f^{^{\prime}}(X_{\tau})Y_{\tau}\rightarrow f^{^{\prime}}(X)Y$ in
probability, hence there exists $\tau_{2}$ such that for all $\tau<\tau_{2}$,
\begin{equation}
\mathbb{P}\left[  \left\vert f^{^{\prime}}(X_{\tau})Y_{\tau}-f^{^{\prime}%
}(X)Y\right\vert \leq\varepsilon\right]  \geq1-\tilde{\varepsilon}.
\label{eq:incr-prob-est}%
\end{equation}
Finally, using the triangle inequality, (\ref{eq:fd-prob-est})\ and
(\ref{eq:incr-prob-est})\ imply
\[
\mathbb{P}\left[  \left\vert \frac{f(X_{\tau}+\tau Y_{\tau})-f(X_{\tau})}%
{\tau}-f^{^{\prime}}(X)Y\right\vert \leq2\varepsilon\right]  \geq
1-2\tilde{\varepsilon},
\]
for all $0<\tau<\min\left\{  \tau_{1},\tau_{2}\right\}  $.
\end{proof}

\begin{theorem}
\label{thm:conv-tau} At each iteration $j$ and time step $i$ of Algorithm
\ref{alg:LM_EnKS2}, $X_{0:i|i}^{j,n,\tau}%
\xrightarrow{\mathrm{P}}%
X_{0:i|i}^{j,n}$ as $\tau\rightarrow0$, where $X_{0:i|i}^{j,n}$ is the $n$-th
member of the ensemble generated at $j$-th iteration in Algorithm
\ref{alg:LM-EnKS-WD} with the same random perturbations as
in Algorithm \ref{alg:LM_EnKS2}.
\end{theorem}

\begin{proof}
In  this proof we omit the subscripts of $N_j$ and $N_{j-1}$. 
The proof is by induction on the number of iterations $j$. For $j=1$ we have
$\bar{X}_{0:i|i}^{j-1,N,\tau}=\bar{X}_{0:i|i}^{j-1,N}$. For
 $j\geq2$, we use induction on time step $i$. For $i=0$ we have
$X_{0|0}^{j,n,\tau}=x_{\mathrm{b}}+V_{b}^{n}=X_{0|0}^{j,n}$. For
$i=1,\ldots,k$, we have from the induction assumption on $i$, $X_{i-1|i-1}%
^{j,n,\tau}%
\xrightarrow{\mathrm{P}}%
X_{i-1|i-1}^{j,n}$ as $\tau\rightarrow0$. Then using Lemma \ref{lem:fd-conv},
we have in (\ref{eq:LM-EnKS2-adv}) as $\tau\rightarrow0$,
\begin{align}
X_{i|i-1}^{j,n,\tau}=  &  \frac{\mathcal{M}_{i}\left(  \bar{X}_{i-1|k}%
^{j-1,N_{},\tau}+\tau\left(  X_{i-1|i-1}^{j,n,\tau}-\bar{X}_{i-1|k}%
^{j-1,N_{},\tau}\right)  \right)  -\mathcal{M}_{i}\left(  \bar{X}%
_{i-1|k}^{j-1,N_{},\tau}\right)  }{\tau}\label{eq:LM-EnKS2-ana-conv}\\
&  +\mathcal{M}_{i}\left(  \bar{X}_{i-1|k}^{j-1,N_{},\tau}\right)  +\mu
_{i}\nonumber\\%
\xrightarrow{\mathrm{P}}%
&  \mathcal{M}_{i}^{\prime}\left(  \bar{X}_{i-1|k}^{j-1,N_{}}\right)  \left(
X_{i-1|i-1}^{j,n}-\bar{X}_{i-1|k}^{j-1,N_{}}\right)  +\mathcal{M}_{i}\left(
\bar{X}_{i-1}^{j-1,N}\right)  +\mu_{i}+V_{i}^{n}=X_{i|i-1}^{j,n}.
\end{align}
Similarly, using the induction assumption on $j$ and Lemma \ref{lem:fd-conv},
we have in (\ref{eq:LM-EnKS2-ana}) and in (\ref{eq:LM-EnKS2-hn}),
respectively,%
\begin{align*}
&  \frac{\mathcal{H}_{i}\left(  \bar{X}_{i|k}^{j-1,N,\tau}+\tau\left(
X_{i|i-1}^{j,n,\tau}-\bar{X}_{i|k}^{j-1,N_{},\tau}\right)  \right)
-\mathcal{H}_{i}\left(  \bar{X}_{i|k}^{j-1,N_{},\tau}\right)  }{\tau}\\
&  \quad%
\xrightarrow{\mathrm{P}}%
\mathcal{H}_{i}^{\prime}\left(  \bar{X}_{i|i}^{j-1,N}\right)\left(  X_{i|i-1}%
^{j,n,\tau}-\bar{X}_{i|k}^{j-1,N_{},\tau}\right),\\
&  \frac{\mathcal{H}_{i}\left(  \bar{X}_{i|k}^{j-1,N,\tau}+\tau\left(
X_{i|i-1}^{j,n,\tau}-\bar{X}_{i|i-1}^{j-1,N_{},\tau}\right)  \right)
-\mathcal{H}_{i}\left(  \bar{X}_{i|k}^{j-1,N_{},\tau}\right)  }{\tau}\\
&  \quad%
\xrightarrow{\mathrm{P}}%
\mathcal{H}_{i}^{\prime}\left(  \bar{X}_{i|k}^{j-1,N_{}}\right)
\left(X_{i|i-1}^{j,n,\tau}-\bar{X}_{i|-1}^{j-1,N_{},\tau}\right).%
\end{align*}
as $\tau\rightarrow0$. In (\ref{eq:h_n}) gives
\begin{equation}
h_{i}^{j,n,\tau}%
\xrightarrow{\mathrm{P}}%
h_{i}^{j,n}\text{ as }\tau\rightarrow0. \label{eq:LM-EnKS2-hn-conv}%
\end{equation}
Using (\ref{eq:LM-EnKS2-hn-conv}) and the continuous mapping theorem in
(\ref{eq:LM-EnKS2-HPHt}) and (\ref{eq:LM-EnKS2-PHt}) gives
\begin{align*}
{P}_{0:i|i-1}^{j,N_{},\tau}\tilde{H}_{i}^{j,\tau\mathrm{T}}%
\xrightarrow{\mathrm{P}}%
{P}_{0:i|i-1}^{j,N_{}}\tilde{H}_{i}^{j\mathrm{T}}\text{\quad as
}\tau &  \rightarrow0,\\
\tilde{H}_{i}^{j,\tau}{P}_{0:i|i-1}^{j,N_{},\tau}\tilde{H}_{i}^{j,\tau\mathrm{T}}%
\xrightarrow{\mathrm{P}}%
\tilde{H}_{i}^{j}{P}_{0:i|i-1}^{j,N_{}}\tilde{H}_{i}^{j\mathrm{T}%
}\text{\quad as }\tau &  \rightarrow0.
\end{align*}
Using also (\ref{eq:LM-EnKS2-ana-conv}) in (\ref{eq:LM-EnKS2-ana}) and the
continuous mapping theorem once more gives $X_{0:i|i}^{j,n,\tau}%
\xrightarrow{\mathrm{P}}%
X_{0:i|i}^{j,n}$ as $\tau\rightarrow0$.
\end{proof}

\begin{corollary}
\label{cor:composite-limit} For each $j$, $\lim_{\min\{N_{1},\ldots
,N_{j}\}\rightarrow\infty}\lim_{\tau\rightarrow0}\bar{X}_{0:k|k}^{j,N,\tau
}=x^{j}$ in probability, where $x^{j}$ is the $j$-th iterate of Algorithm
\ref{alg:LM_D}.
\end{corollary}

\begin{proof}
The proof follows immediately from Theorem \ref{thm:conv-tau}, Theorem
\ref{thm:conv-N}, and Lemma \ref{lem:U_eq_min_r}.
\end{proof}

\section{Conclusion}

In this paper we have shown that: when the observation and the model operators are linear for any time step,
the empirical mean and covariance of EnKS converge
 to the KS mean and covariance in the limit for large ensemble size  in $L^{p}$ for any
 $p \in [1,\infty)$.  In the nonlinear case, i.e., in the case where the observation and the model operators are not necessary linear, we have shown the convergence
of LM-EnKS iterations (Algorithm~\ref{alg:LM_EnKS2}) in the limit for large ensemble size. The convergence is in the sense that (i) each
 iterate generated by Algorithm~\ref{alg:LM_EnKS2} converges in probability to its corresponding iterate
 of Algorithm~\ref{alg:LM-EnKS-WD} as the finite differences parameter goes to zero, (ii) and that each iterate generated by
 Algorithm~ \ref{alg:LM-EnKS-WD} converges, in $L^{p}$ for any
 $p \in [1,\infty)$, to its corresponding iterate of Algorithm~\ref{alg:LM_D} (the Levenberg-Marquardt algorithm) in the large-ensemble limit.

 These proofs of convergence, and more generally the asymptotic behavior of
 the ensemble-based algorithms  deserve further investigation. Here in the nonlinear case,
 we have given only the limit in probability of each iterate of  Algorithm~\ref{alg:LM_EnKS2} as the finite differences parameter goes to zero
 and the ensemble sizes go to infinity.
  One may, for instance, try to prove stronger convergence results, especially
  to leverage the convergences in probability to convergences in $L^{p}$,  and show the convergence rate of these algorithms following the spirit of \cite{LeGland-2011-LSA}.
The approach followed in this paper could be also extended to the case in which  other variants of ensemble method, such as the square root ensemble Kalman filter \cite{Kwiatkowski-2014-CSR},
  are used to approximately solve the linearized subproblem. 

\bibliographystyle{siam}
\bibliography{../../references/geo,../../references/other}

\newcommand{\noopsort}[1]{}
\begin{thebibliography}{10}

\bibitem{Anderson-1979-OF}
{\sc Brian D.~O. Anderson and John~B. Moore}, {\em Optimal filtering},
  Prentice-Hall, Englewood Cliffs, N.J., 1979.

\bibitem{Bell-1993-IKF}
{\sc B.~M. Bell and F.~W. Cathey}, {\em The iterated {K}alman filter update as
  a {G}auss-{N}ewton method}, IEEE Transactions on Automatic Control, 38
  (1993), pp.~294--297.

\bibitem{Billingsley-1995-PM}
{\sc Patrick Billingsley}, {\em Probability and measure}, John Wiley \& Sons
  Inc., New York, third~ed., 1995.

\bibitem{Bocquet-2012-CII}
{\sc M.~Bocquet and P.~Sakov}, {\em Combining inflation-free and iterative
  ensemble {K}alman filters for strongly nonlinear systems}, Nonlinear
  Processes in Geophysics, 19 (2012), pp.~383--399.

\bibitem{Bocquet-2014-IEK}
\leavevmode\vrule height 2pt depth -1.6pt width 23pt, {\em An iterative
  ensemble {K}alman smoother}, Quarterly Journal of the Royal Meteorological
  Society, 140 (2014), pp.~1521--1535.

\bibitem{Burgers-1998-ASE}
{\sc Gerrit Burgers, Peter~Jan van Leeuwen, and Geir Evensen}, {\em Analysis
  scheme in the ensemble {K}alman filter}, Monthly Weather Review, 126 (1998),
  pp.~1719--1724.

\bibitem{Courtier-1994-SOI}
{\sc P.~Courtier, J.-N. Th\'{e}paut, and A.~Hollingsworth}, {\em A strategy for
  operational implementation of {4D-Var}, using an incremental approach},
  Quarterly Journal of the Royal Meteorological Society, 120 (1994),
  pp.~1367--1387.

\bibitem{Evensen-2009-DAE}
{\sc Geir Evensen}, {\em Data Assimilation: The Ensemble {K}alman Filter},
  Springer, 2nd~ed., 2009.

\bibitem{Fisher-2005-EKS}
{\sc M.~Fisher, M.~Leutbecher, and G.~A. Kelly}, {\em On the equivalence
  between {K}alman smoothing and weak-constraint four-dimensional variational
  data assimilation}, {Quarterly Journal of the Royal Meteorological Society},
  {131} ({2005}), pp.~{3235--3246}.

\bibitem{Gill-1978-ASN}
{\sc Philip~E. Gill and Walter Murray}, {\em Algorithms for the solution of the
  nonlinear least-squares problem}, SIAM J. Numer. Anal., 15 (1978),
  pp.~977--992.

\bibitem{Hamill-2000-HEK-x}
{\sc Thomas~M. Hamill and Chris Snyder}, {\em A hybrid ensemble {K}alman
  filter--3{D} variational analysis scheme}, Monthly Weather Review, 128
  (2000), pp.~2905--2919.

\bibitem{Johns-2008-TEK}
{\sc Craig~J. Johns and Jan Mandel}, {\em A two-stage ensemble {K}alman filter
  for smooth data assimilation}, {E}nvironmental and Ecological Statistics, 15
  (2008), pp.~101--110.

\bibitem{Kalman-1960-NAL}
{\sc Rudolph~Emil Kalman}, {\em A new approach to linear filtering and
  prediction problems}, Transactions of the ASME -- Journal of Basic
  Engineering, Series D, 82 (1960), pp.~35--45.

\bibitem{Kalnay-2010-EKF}
{\sc Eugenia Kalnay}, {\em Ensemble {Kalman} filter: {Current} status and
  potential}, in Data Assimilation, William Lahoz, Boris Khattatov, and Richard
  M\'{e}nard, eds., Springer Berlin Heidelberg, 2010, pp.~69--92.

\bibitem{Kelly-2014-WAE}
{\sc D.~T.~B. Kelly, K.~J.~H. Law, and A.~M. Stuart}, {\em Well-posedness and
  accuracy of the ensemble {K}alman filter in discrete and continuous time},
  Nonlinearity, 27 (2014), pp.~2579--2603.

\bibitem{Khare-2008-IAE}
{\sc Shree~P. Khare, Jeffrey~L. Anderson, Timothy~J. Hoar, and Douglas Nychka},
  {\em An investigation into the application of an ensemble {K}alman smoother
  to high-dimensional geophysical systems}, Tellus A, 60 (2008), pp.~97--112.

\bibitem{Kwiatkowski-2014-CSR}
{\sc Evan Kwiatkowski and Jan Mandel}, {\em Convergence of the square root
  ensemble {Kalman} filter in the large ensemble limit}, SIAM Journal on
  Uncertainty Quantification,  (2014), p.~In print.
\newblock arXiv:1404.4093.

\bibitem{LeDimet-1986-VAA}
{\sc Fran\c{c}ois-Xavier Le~Dimet and Olivier Talagrand}, {\em {Variational
  algorithms for analysis and assimilation of meteorological observations:
  theoretical aspects}}, Tellus A, 38 (1986), pp.~97--110.

\bibitem{LeGland-2011-LSA}
{\sc F.~{Le~Gland}, V.~Monbet, and V.-D. Tran}, {\em Large sample asymptotics
  for the ensemble {K}alman filter}, in The Oxford Handbook of Nonlinear
  Filtering, Dan Crisan and Boris Rozovski\v{\i}, eds., Oxford University
  Press, 2011, pp.~598--631.

\bibitem{Levenberg-1944-MSC}
{\sc Kenneth Levenberg}, {\em A method for the solution of certain non-linear
  problems in least squares}, Quarterly of Applied Mathematics, 2 (1944),
  pp.~164--168.

\bibitem{Li-2001-OVD}
{\sc Zhijin Li and I.~M. Navon}, {\em Optimality of variational data
  assimilation and its relationship with the {K}alman filter and smoother},
  Quarterly Journal of the Royal Meteorological Society, 127 (2001),
  pp.~661--683.

\bibitem{Liu-2013-EFV}
{\sc Chengsi Liu and Qingnong Xiao}, {\em An ensemble-based four-dimensional
  variational data assimilation scheme. {Part III}: {A}ntarctic applications
  with {Advanced Research WRF} using real data}, Monthly Weather Review, 141
  (2013), pp.~2721--2739.

\bibitem{Liu-2008-EFV}
{\sc Chengsi Liu, Qingnong Xiao, and Bin Wang}, {\em An ensemble-based
  four-dimensional variational data assimilation scheme. {Part I: T}echnical
  formulation and preliminary test}, Monthly Weather Review, 136 (2008),
  pp.~3363--3373.

\bibitem{Liu-2009-EFV}
\leavevmode\vrule height 2pt depth -1.6pt width 23pt, {\em An ensemble-based
  four-dimensional variational data assimilation scheme. {Part II}: Observing
  system simulation experiments with {Advanced Research WRF (ARW)}}, Monthly
  Weather Review, 137 (2009), pp.~1687--1704.

\bibitem{Mandel-2013-4EK}
{\sc Jan Mandel, Elhoucine Bergou, and Serge Gratton}, {\em {4DVAR} by ensemble
  {K}alman smoother}.
\newblock arxiv:1304.5271, 2014.

\bibitem{Mandel-2011-CEK}
{\sc Jan Mandel, Loren Cobb, and Jonathan~D. Beezley}, {\em On the convergence
  of the ensemble {K}alman filter}, Applications of Mathematics, 56 (2011),
  pp.~533--541.

\bibitem{Marquardt-1963-ALE}
{\sc Donald~W. Marquardt}, {\em An algorithm for least-squares estimation of
  nonlinear parameters}, Journal of the Society for Industrial and Applied
  Mathematics, 11 (1963), pp.~431--441.

\bibitem{Osborne-1976-NLL}
{\sc M.~R. Osborne}, {\em Nonlinear least squares -- the {L}evenberg algorithm
  revisited}, Journal of the Australian Mathematical Society Series B, 19
  (1976), pp.~343--357.

\bibitem{Sakov-2012-IES}
{\sc Pavel Sakov, Dean~S. Oliver, and Laurent Bertino}, {\em An iterative
  {EnKF} for strongly nonlinear systems}, Monthly Weather Review, 140 (2012),
  pp.~1988--2004.

\bibitem{Simon-2006-OSE}
{\sc Dan Simon}, {\em Optimal State Estimation: {K}alman, $\mathrm{H}_\infty$,
  and Nonlinear Approaches}, John Wiley and Sons, 2006.

\bibitem{Tremolet-2007-ME4}
{\sc Yannick Tr\'{e}molet}, {\em Model-error estimation in {4D-Var}}, Quarterly
  Journal of the Royal Meteorological Society, 133 (2007), pp.~1267--1280.

\bibitem{Tshimanga-2008-LPA}
{\sc J.~Tshimanga, S.~Gratton, A.~T. Weaver, and A.~Sartenaer}, {\em
  Limited-memory preconditioners, with application to incremental
  four-dimensional variational data assimilation}, Quarterly Journal of the
  Royal Meteorological Society, 134 (2008), pp.~751--769.

\bibitem{vanderVaart-2000-AS}
{\sc A.~W. {V}an~der Vaart}, {\em Asymptotic Statistics}, Cambridge University
  Press, 2000.

\bibitem{Wang-2010-IEC}
{\sc Xuguang Wang}, {\em Incorporating ensemble covariance in the gridpoint
  statistical interpolation variational minimization: {A} mathematical
  framework}, Monthly Weather Review, 138 (2010), pp.~2990--2995.

\bibitem{Zupanski-2005-MLE}
{\sc Milija Zupanski}, {\em Maximum likelihood ensemble filter: {T}heoretical
  aspects}, Monthly Weather Review, 133 (2005), pp.~1710--1726.

\end{thebibliography}

\end{document}